\documentclass[11pt, a4paper]{article}
\usepackage{amsmath,amssymb,amsfonts}
\usepackage{amsthm}
\usepackage{verbatim}
\usepackage{a4wide}
\usepackage{epsfig,epic,eepic,graphicx}
\usepackage{graphicx}
\usepackage[active]{srcltx}

\newtheorem{theorem}{Theorem}
\newtheorem{lemma}[theorem]{Lemma}
\newtheorem{proposition}[theorem]{Proposition}

\newtheorem{definition}[theorem]{Definition}

\newtheorem{remark}[theorem]{Remark}

\makeatletter

\@addtoreset{equation}{section}
\makeatother

\newcommand{\T}{{\mathbb T}}
\newcommand{\N}{{\mathbb N}}
\newcommand{\Z}{{\mathbb Z}}
\newcommand{\Q}{{\mathbb Q}}
\newcommand{\R}{{\mathbb R}}

\newcommand{\pa}{{\partial}}
\newcommand{\na}{{\nabla}}
\newcommand{\dis}{{\displaystyle}}
\newcommand{\eps}{{\varepsilon}}
\newcommand{\be}{\begin{equation}}
\newcommand{\ee}{\end{equation}}
\newcommand{\essinf}{\mathrm{ess inf}}
\newcommand{\ueapp}{u^\varepsilon_\text{app}}
\newcommand{\re}{r^\eps}
\newcommand{\Gae}{\Gamma^\eps}
\newcommand{\Om}{\Omega}

\newcommand{\dv}{\mathrm{div}\;}
\newcommand{\om}{\omega}
\newcommand{\we}{w^\eps}
\newcommand{\ue}{u^\eps}
\newcommand{\qe}{q^\eps}

\def\div{\hbox{div  }}

\def\dis{\displaystyle}

\title{Effective boundary condition at a rough surface \\
starting  from a slip condition} 
\author{\footnote{DMA/CNRS, Ecole Normale Sup\'erieure, 45 rue d'Ulm, 75005
   Paris, FRANCE, tel: +33 1 44 32 20 58, fax: +33 1 44 32 20 80 (corresponding author).} Anne-Laure Dalibard 
and 
 \footnote{IMJ, 175 rue du Chevaleret, 75013 Paris. } David G\'erard-Varet }
\date{}

\begin{document}
\maketitle

\begin{abstract}
We consider the homogenization of the  Navier-Stokes equation, set  in a
channel with a  rough boundary,  of small amplitude and
wavelength $\epsilon$. It was shown recently  that, for any  
non-degenerate roughness pattern, and for any  reasonable condition  imposed at the rough boundary,  the
homogenized boundary condition in the limit $\eps = 0$ is always
 no-slip.   We give in this paper error estimates for this homogenized
 no-slip condition,  and provide a  more accurate effective  boundary
 condition, of Navier type. Our result extends those obtained in
 \cite{Basson:2007,GerMas}, in which  the special  case of a  Dirichlet
 condition at the rough boundary  was examined. 
\end{abstract}

{{\em Keywords:} \small Wall laws, rough boundaries, homogenization,
  ergodicity, Korn inequality} 
 
\section{Introduction}
Most works on Newtonian liquids  assume the
validity of the no-slip boundary condition: the  velocity field of the
liquid  at a solid surface equals the velocity field of the surface itself. 
This assumption relies  on both theoretical  and experimental studies,
carried over more than a century. 

\medskip
Still,  with the  recent surge of activity around microfluidics, the
question of fluid-solid interaction has been reconsidered, and the
consensus  around the no-slip condition has been questioned. Several
 experimentalists, observing for instance water over mica, have reported
 significant slip. More generally, it has been claimed that, in many
 cases,  the liquid velocity  field $u$ obeys a Navier condition at
 the solid boundary $\Sigma$:  
\begin{equation} \label{Navier} \tag{Na}
 (I_d - \nu \otimes  \nu ) u\vert_{\Sigma}  \:  =  \:  \lambda   (I_d - \nu
 \otimes  \nu ) D(u)\nu \vert_{\Sigma},  
 \quad u \cdot \nu\vert_{\Sigma}  \: = \:  0,  \quad \lambda > 0 
\end{equation}
 where $\nu$ is an inward normal vector to $\Sigma$, and $D(u)$ is the
 symmetric part of the gradient.  Slip lengths
 $\lambda$ up to a few micrometers have been measured. This is far more than the
 molecular scale, and would therefore invalidate the (macroscopic) no-slip condition 
\begin{equation} \label{Dirichlet} \tag{Di}
u\vert_{\Sigma}  = 0. 
\end{equation}
Nevertheless,  such experimental  results are  widely debated. 
For similar experimental settings,  there are huge discrepancies
between the measured values of $\lambda$.  We refer to the article
\cite{Lauga:2007}  for an overview. 

\medskip
In this debate around boundary conditions, the irregularity of the solid
surface is a major issue. Again, its effect is a topic of intense
discussion. On one hand, some people argue that it  increases the surface of
friction, and may cause a decrease of the slip. On the other
hand, it may generate small scale phenomena favourable to slip. For
instance, some rough hydrophobic surfaces seem more slippery due to the
trapping of air bubbles in the humps of the roughness. Moreover,
irregularity creates a boundary layer in its vicinity, meaning high
velocity gradients. Thus, even though \eqref{Dirichlet}  is satisfied
at the rough boundary, there may be significant velocities right above. In
other words, the no-slip condition may hold at the {\em small} scale of the
boundary layer but not at the {\em large} scale of the mean flow. This
phenomenon, due to scale separation, is called {\em apparent slip} in
the physics litterature. 

\medskip
In parallel to experimental works, several theoretical studies have  
been carried, so as to clarify the role of roughness. Many 
of them relate to homogenization theory. First,  the irregularity is modeled  by
small-scale variations of the boundary. Then, an
asymptotic analysis is performed, as  the small scales go to zero.  The idea is to
replace the {\em constitutive boundary condition}  at the  {\em rough}
surface by a homogenized or {\em effective boundary condition} at the
smoothened surface.  In this way, one can describe the averaged
effect of the roughness.  We stress that such homogenized conditions
(often called {\em wall laws}) are also  of
practical interest in numerical codes. They allow to filter out the
small scales of the boundary, which have a high
computational cost.

\medskip
Let us recall briefly the main mathematical results on wall
laws. To give a unified description,  we take a single
model. Namely, we consider a  two-dimensional rough channel 
$$\Omega^\eps \: :=  \: \Omega \cup \Sigma \cup R^\eps $$
where $\Omega = \R \times (0,1)$ is the {\em smooth part}, $R^\eps$ is the
rough part, and $\Sigma = \R \times \{0\}$ their interface. We assume that
the rough part has typical size $\eps$, that is 
$$ \qquad   R^\eps \: := \: \eps R, \quad  R \: :=  \: \left\{ y, \: 0 >  y_2 >
\omega\left(y_1\right) \right\} $$ 
for a  Lipschitz function $\omega : \R \mapsto (-1,0)$. We also
introduce 
$$ \Gamma^\eps \: := \:  \eps \, \Gamma, \quad \Gamma \: := \: \left\{
  y, \:   y_2 =
\omega\left(y_1\right) \right\} $$ 
See Figure \ref{fig1} for  notations. 
\begin{figure} \label{fig1}
\begin{center}
\includegraphics[height = 5.5cm,width=8.5cm]{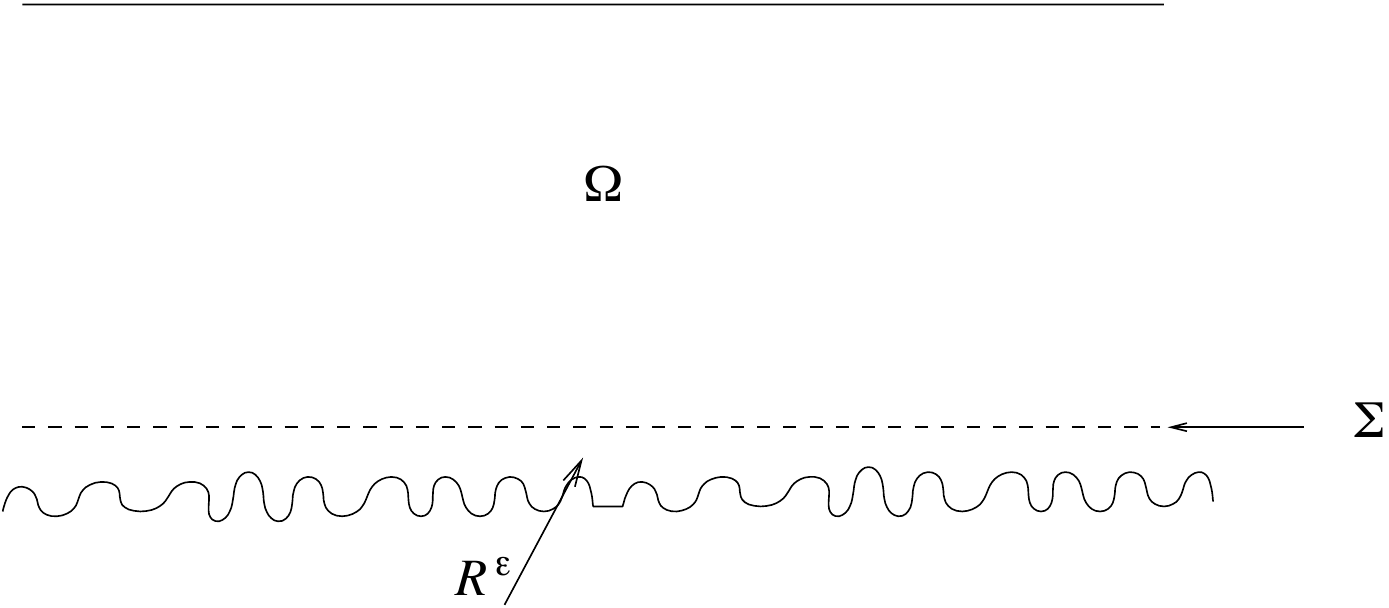}
\end{center}
\caption{The rough domain $\Omega^\eps$.}
\end{figure}
 We consider in this channel a steady flow $u^\eps$. It is modeled by the  stationary 
 Navier-Stokes system, with a prescribed flux $\phi$ across a vertical
 cross-section   $\sigma^\eps$ of $\Omega^\eps$. Moreover, to cover
 all interesting cases, we shall  consider
 either pure slip, partial slip  or no-slip at the rough boundary
 $\Gamma^\eps$. This  means that the
 constant $\lambda$ below shall be either $+\infty$,  positive or zero. For
 simplicity, we assume no-slip at the upper boundary. We get eventually
\begin{equation} \label{NSeps} \tag{NS$^\eps$}
\left\{ 
\begin{aligned}
 - \Delta u^\eps + u^\eps \cdot \na u^\eps + \na p^\eps  &  
= 0, \: x \in \Omega^\eps, \\
\div u^\eps  &   = 0, \: x \in \Omega^\eps, \\
 u^\eps\vert_{x_2 = 1}    = 0,  \: \int_{\sigma^\eps} u^\eps_1  &  =
 \phi, \\
(I_d - \nu \otimes  \nu ) u^\eps\vert_{\Gamma^\eps}  \:    =  \:  \lambda^\eps   (I_d
- & \nu \otimes  \nu )  D(u^\eps)\nu\vert_{\Gamma^\eps}, \: u^\eps \cdot \nu\vert_{\Gamma^\eps} \:
= \:  0. 
\end{aligned}
\right.
\end{equation}
Notice  that the flux integral in the third equation does
not depend on the location  of the cross-section $\sigma^\eps$, thanks to the
divergence-free and impermeability conditions.   
We also emphasize that this problem has a  singularity  in $\eps$, due
to the high  frequency oscillation of the boundary. 
Thus,  the problem is to  {\em replace the singular
  problem in $\Omega^\eps$ by a regular  problem in $\Omega$}. The idea
 is to  keep the same Navier-Stokes equations  
 \begin{equation} \label{NS} \tag{NS}
\left\{
\begin{aligned}
- \Delta u + u \cdot \na u + \na p  & = 0, \: x \in \Omega, \\
\div u  & = 0, \: x \in \Omega, \\
u\vert_{x_2=1}  = 0,  \quad \int_{\sigma} u_1 & = \phi,
\end{aligned} 
\right.
\end{equation}
but  with a   boundary condition at the artificial boundary
$\Sigma$ which is regular in $\eps$.  The problem  is to find the most
accurate  condition.

\medskip
A series of papers has addressed this problem, starting from the standard
Dirichlet condition  at $\Gamma^\eps$ ($\lambda^\eps = 0$ in \eqref{NSeps}). 
Losely, two main facts have been established:
 \begin{enumerate}
\item {\em For any roughness profile $\omega$, the Dirichlet condition \eqref{Dirichlet}
  provides a $O(\eps)$ approximation of $u^\eps$ in
  $L^2_{uloc}(\Omega)$.} 
\item {\em For generic roughness profile $\omega$, the Navier condition
does better, choosing $\lambda =
  \alpha \eps$ for some good constant $\alpha$ in \eqref{Navier}}.
\end{enumerate} 
Of course,  such statements are only the crude translations of
cumulative rigorous results. Up to our knowledge, the pioneering results
on wall laws are due to  Achdou, Pironneau and Valentin \cite{Achdou:1995,Achdou:1998}, and
J\"ager and Mikelic \cite{Jager:2001,Jager:2003}, who considered periodic roughness
profiles $\omega$.  See also \cite{Amirat:2001a} on this periodic case. The
extension to arbitrary roughness profiles has been studied by the
second author (and coauthors) in articles \cite{Basson:2007,DGV:2008,GerMas}. The expression {\em
  generic roughness profile} means functions $\omega$ with
ergodicity properties (for instance, $\omega$ is random stationary, or
almost periodic). We refer to the forementioned works for all details and
rigorous statements.  Let us just mention that the slip length $\alpha
\,  \eps$ is related to  a boundary layer  of amplitude $\eps$ near
the rough boundary. It is the mathematical expression of the 
apparent slip discussed earlier. 

\medskip
Beyond the special  case $\lambda^\eps = 0$, some studies  have dealt with 
the general case  $\lambda^\eps \in [0,+\infty]$.  The limit
$u^0$  of $u^\eps$,  and the   condition that it satisfies at $\Sigma$ have been
 investigated. In brief, the  striking conclusion  of these studies  is
 that, {\em as soon as the boundary is genuinely rough, $u^0$
   satisfies a no-slip condition at $\Sigma$}.  This idea has been
 developped in \cite{Simon:2003} for a periodic roughness pattern $\omega$. It
 has been  generalized to arbitrary roughness pattern  in \cite{Bucur:2008}. In this last
 article, the assumption of {\em genuine roughness} is expressed in
 terms of Young measure.  When recast in our 2D setting, it reads:
\begin{center}
{\em (H)   The family of Young measures
  $\bigl(d\mu_{y_1}\bigr)_{y_1}$   associated with the  sequence  
  $\bigl(\omega'(\cdot/\eps)\bigr)_\eps$ is s.t.
$$  d\mu_{y_1} \neq \delta_{0}  \mbox{ (the Dirac mass at zero), for
  almost every }  \: y_1 \in \R.$$  }
\end{center}
Under (H),  one can show  that $u^\eps$ locally converges in $H^1$-weak
to the famous Poiseuille flow: 
\begin{equation*}
u^0(x) \: = \: \left(U^0(x_2),0\right), \quad U^0(x_2) \: = \:  6 \phi x_2
(1-x_2) 
\end{equation*} 
which is solution of \eqref{NS}-\eqref{Dirichlet}.  We refer to
\cite{Bucur:2008} for all details.  

\medskip
This result can be seen as a mathematical justification of the no-slip
condition. Indeed, any realistic boundary is rough. If one is only 
interested in  scales greater than the scale $\eps$ of the roughness,
then \eqref{Dirichlet} is an appropriate  boundary condition, whatever the
microscopic phenomena behind. Still, as in the case $\lambda^\eps = 0$,  
one may be interested in more quantitative estimates. How good is the
boundary condition? Can it be improved? Is there possibility of a
$O(\eps)$ slip? Such questions are especially important in
microfluidics,  a domain in which minimizing wall friction is
crucial (see \cite{Bocquet:2007}).

\medskip
The aim of the present article is to address these questions. We shall
extend to an arbitrary slip length $\lambda^\eps$ the kind of  results obtained for
$\lambda^\eps = 0$.  Of course, as in the works  mentioned above, we must
assume some non-degeneracy of the roughness pattern. 
We make the following assumption:
\begin{center}
{\em (H') There exists ${\cal C} > 0$, such that  for all 2-D fields $u \in
  C^\infty_c\left(\overline{R}\right)$ satisfying $u\cdot \nu\vert_{\Gamma} =
  0$,
$$     \| u \|_{L^2(R)} \: \le \:  {\cal C}\, \| \na u \|_{L^2(R)} . $$
}
\end{center}
Assumption (H'), and its relation to the assumption (H) will be
discussed thoroughly in the next section. 
Broadly, we obtain two main results. The first one is 
\begin{theorem}  \label{Direstimates}
There exists $\phi_0 > 0$, such that for  all $|\phi| < \phi_0$, for all 
 $\eps \: \le \:  1$, system \eqref{NSeps} has a unique solution $u^\eps$ in
$H^1_{uloc}(\Omega^\eps)$. Moreover, if $\lambda^\eps = 0$ or
if  (H') holds, one has 
$$ \| u^\eps - u^0 \|_{H^1_{uloc}(\Omega^\eps)} \:  \le \:  C \, \phi
\,  \sqrt{\eps}, \quad \| 
u^\eps - u^0 \|_{L^2_{uloc}(\Omega)} \: \le \:    C \,  \phi \, \eps, $$
where $u^0$ is the Poiseuille flow, satisfying \eqref{NS}-\eqref{Dirichlet}. 
\end{theorem}
In short, the Dirichlet wall law provides a $O(\phi \, \eps)$ approximation of the
exact solution $u^\eps$ in
$L^2_{uloc}(\Omega)$, for any $\lambda^\eps \in [0,+\infty]$.  This gives a
quantitative estimate of the convergence results obtained in the
former papers. Note that the dependence of  the error
estimates on both  $\phi$ and $\eps$ is specified. In the case
$\lambda^\eps = 0$, this improves slightly the result of \cite{Basson:2007}, where
the $\phi$ dependence was neglected.  

\medskip
Our second result is  the existence of a better homogenized
condition. Here, as outlined in article \cite{GerMas}, some ergodicity
property of the rugosity is needed.  We shall assume that $\omega$ is
a random stationary process. Moreover, we shall need a
slight reinforcement of (H'), namely: 
\begin{center}
{\em (H'') There exists ${\cal C} > 0$, such that  for all 2-D fields $u \in
  C^\infty_c\left(\overline{R}\right)$ satisfying $u\cdot \nu\vert_{\Gamma} =
  0$,
$$     \| u \|_{L^2(R)} \: \le \:  {\cal C}\, \| D(u) \|_{L^2(R)} ,
\quad D(u) = \frac{1}{2} \left( \na u + (\na u)^t \right).   $$
}
\end{center}
We shall discuss this assumption in section \ref{sec1}. We state
\begin{theorem} \label{Navestimates2}
Let $\omega$ be an ergodic stationary random process,  with
values in $(-1,0)$ and  $K$-Lipschitz almost surely, for some $K
>0$. Assume  either that $\lambda^\eps = 0$, or that $\lambda^\eps=\lambda^0>0$ for all $\eps$, and the
non-degeneracy condition (H'') holds almost surely, with a uniform
${\cal C}$.  Then there exists $\alpha > 0$   and $\phi_0 > 0$  such
that, for all  $|\phi| < \phi_0$, $\: \eps \le 1$, the solution 
$u^N$ of \eqref{NS}-\eqref{Navier} with $\lambda = \alpha \, \eps$
satisfies 
$$ \left( \sup_{R \ge 1} \frac{1}{R}  \int_{\Omega \cap  \{|x_1| < R\}} |
u^\eps - u^N |^2 \, dx \right)^{1/2} = o(\eps), \:\mbox{ almost surely.}  $$   
\end{theorem}
We quote that the norm above is common in the framework of stochastic
pde's: see for instance  \cite{Basson}. We also quote that, even in the case $\lambda^\eps = 0$,
this almost sure estimate is new:  the estimates of \cite{Basson:2007}
involved expectations.   
This result can also be extended to other slip lengths $\lambda^\eps$ in \eqref{NSeps}; more precisely, up to a few minor modifications, our techniques also allow us to treat slip lengths $\lambda^\eps$ such that $\lambda^\eps \gg 1$, or $\lambda^\eps\lesssim \eps^2$, or $\lambda^\eps=\lambda^0\eps$.

\medskip
Briefly, the outline of the paper is as follows. In section  \ref{sec1}, we
will discuss in details the hypotheses (H') and (H''). Section \ref{sec2} will
be devoted to the proof of theorem \ref{Direstimates}. In section \ref{sec3}, we will
analyze the boundary layer near the rough
boundary. This will allow for  the proof
of Theorem  \ref{Navestimates2}, to be achieved
in section  \ref{sec4}.

\section{The non-degeneracy assumption} \label{sec1}

The goal of this section is to discuss hypotheses (H') and (H''), and,
in particular, to give sufficient conditions on the function $\omega$
for (H') and (H'') to hold. We will also discuss the optimality of
these conditions in the periodic, quasi-periodic and stationary
ergodic settings, and compare them to assumption (H). 

\subsection{Poincar\'e inequalities for rough domains: assumption (H')}

First, let us recall that if the non-penetration condition $u\cdot \nu
\vert_\Gamma=0$ is replaced by a no-slip condition $u\vert_\Gamma=0$,
then the Poincar\'e inequality holds: indeed, for all $u\in H^1(R)$
such that $u\vert_\Gamma=0$, we have 
\begin{eqnarray*}
\int_R  | u(y_1,y_2)|^2 dy_1 \:dy_2&=& \int_R \left|
\int_{\omega(y_1)}^{y_2}\pa_2 u (y_1,t)dt\right|^2 dy_1\: dy_2\\ 
&\leq & C \int_R |\pa_2 u(y_1,t)|^2 dy_1\: dt,
\end{eqnarray*}
where the constant $C$ depends only on $\|\omega\|_{L^\infty}$. 

Assumption (H') requires that the same inequality holds under the mere
non-penetration condition; of course, such an inequality is false in
general (we give a counter-example below in the case of a flat
bottom). In fact, (H') is strongly related to the roughness of the
boundary: if the function $\omega$ is not constant, then the inward normal
vector $\nu=(1+  {\omega'}^2)^{-1/2}(-\omega', 1)$ takes different
values. Since $u\cdot \nu \vert_\Gamma=0$, we have a control of $u$ in
several directions at the boundary (at different points of
$\Gamma$). \textit{In fine}, this allows us to prove that the
Poincar\'e inequality holds, and the arguments are in fact close to
the calculations of the Dirichlet case recalled above. 

$\bullet$ We now derive a \textbf{sufficient condition for (H'):}

\begin{lemma}
Let $\omega\in W^{1, \infty}(\R)$ with values in $(-1,0)$ and such
that $\sup \omega<0$. Assume that 
\be\label{hyp:nondegenerate}
\exists A>0,\ \inf_{y_1\in\R} \int_{0}^{A} |\omega'(y_1+t)|^2 dt >0.
\ee

Then assumption (H') is satisfied.

\label{lem:Poincare}
\end{lemma}

\begin{proof}
 The idea is to prove that for some well-chosen number $B>0$, there holds
\begin{eqnarray}
\int_R |u(y)|^2 dy &\leq &C_B \int_R \int_0^B 	\left| u(y_1,
y_2)\cdot \nu(y_1+ t)\right|^2\:dt \:dy_1\:dy_2 \label{in:ndg1}\\ 
&\leq & C_B \int_R |\nabla u(y)|^2\:dy.\label{in:ndg2}
\end{eqnarray}
The first inequality is a direct consequence of assumption
\eqref{hyp:nondegenerate}. The proof of the second one follows
arguments from \cite{Simon:2003}, and is in fact close to the proof of the
Poincar\'e inequality in the Dirichlet case. 

First, for all $B>0$, we have
\begin{eqnarray*}
&&\int_R \int_0^B 	\left| u(y_1, y_2)\cdot \nu (y_1+
  t)\right|^2\:dt \:dy_1\:dy_2\\ 
 &=&	\int_R \int_0^B \frac{1}{1+ {\omega'}^2(y_1+t)}	\left( -u_1(y)
  \omega'(y_1+t) +u_2(y)\right)^2\:dt \:dy_1\:dy_2\\ 
&\geq & \frac{1}{1+ \|\omega'\|_\infty^2}\left[\int_R
    dy\:u_1^2(y)\int_0^B\omega'(y_1+t)^2 dt+ B\int_R u_2^2\right.\\ 
&&\qquad\qquad \qquad \left. - 2\int_R u_1(y) u_2(y)(\omega(y_1+ B) -
    \omega(y_1))\:dy\right]\\  
&\geq &  \frac{1}{1+ \|\omega'\|_\infty^2} \inf\left( B-
  \|\omega\|_\infty, \inf_{y_1\in \R} \int_0^B\omega'(y_1+t)^2 dt -
  \|\omega\|_\infty \right) \int_R |u(y)|^2 \:dy. 
\end{eqnarray*} 
Assume that $B>A$, and set 
$$
\alpha:=\inf_{y_1\in\R} \int_{0}^{A} |\omega'(y_1+t)|^2 dt.
$$
Notice that $\alpha>0$ thanks to \eqref{hyp:nondegenerate}. Then 
$$
\inf_{y_1\in \R} \int_0^B\omega'(y_1+t)^2 dt\geq
\left\lfloor\frac{B}{A}\right\rfloor \alpha,  
$$
and thus there exists a positive constant $c$ such that for all $B>A$,
$$
\int_R \int_0^B 	\left| u(y_1, y_2)\cdot \nu (y_1+
t)\right|^2\:dt \:dy_1\:dy_2\geq c (B-1)\int_R |u(y)|^2 \:dy. 
$$ 
Thus for $B$ large enough, inequality \eqref{in:ndg1} is satisfied.

As for \eqref{in:ndg2}, let us now prove that for all $B>0$, there
exists a constant $C_B$ such that 
$$
\int_R \int_0^B \left| u(y_1, y_2)\cdot \nu (y_1+ t)\right|^2\:dt
\:dy_1\:dy_2 \leq C_B\int_R |\nabla u(y)|^2\:dy. 
$$
We use the same kind of calculations as in \cite{Simon:2003}. The idea is the
following: for all $y\in R$, $t\in[0,B]$, let 
$$
z=(y_1+ t, \omega(y_1+t))\in \Gamma.
$$
Let $\ell_{y,t}$ be a path in $ W^{1,\infty}([0,1], \R^2)$ such that
$\ell_{y,t}(0)=y,$ $\ell_{y,t}(1)=z$ and $\ell_{y,t}(\tau)\in R$ for
all $\tau\in(0,1)$. Then 
$$ 
u(y)-u(z)=\int_0^1 \left(\ell_{y,t}'(\tau)\cdot \nabla\right)
u(\ell_{y,t}(\tau))\:d\tau, 
$$
and thus, since $u(z)\cdot \nu(y_1+t)=0$,
$$
\left|u(y)\cdot \nu(y_1 + t)  \right|\leq
\int_0^1\left|\left(\ell_{y,t}'(\tau)\cdot \nabla\right)
u(\ell_{y,t}(\tau)) \right|\:d\tau\:dt\:dy. 
$$
There remains to choose a particular path $\ell_{y,t}$.

Notice that in general, we cannot choose for  $\ell_{y,t}$ the
straight line joining $y$ and $z$, since the latter may cross the
boundary $\Gamma$. We thus make the following choice: for
$\lambda\in(\sup \omega, 0)$, we set 
$$
\begin{aligned}
z_\lambda':=(y_1, \lambda),\\
z_\lambda'':=(y_1+t, \lambda).	
\end{aligned}
$$
We define the path $\ell_{y,t}$ by 
$$
\ell_{y,t}(0)=y,\quad\ell_{y,t}\left(\frac{1}{3}
\right)=z_\lambda',\quad\ell_{y,t}\left(\frac{2}{3}
\right)=z_\lambda'', \quad\ell_{y,t}\left(1 \right)=z, 
$$  
and $\ell_{y,t}$ is a straight line on each segment $[0,1/3]$, $[1/3,
  2/3]$, $[2/3, 1]$ (see Figure \ref{fig:chemin}).

\begin{figure} 
\begin{center}
\includegraphics[height = 6.5cm, width=10.5cm]{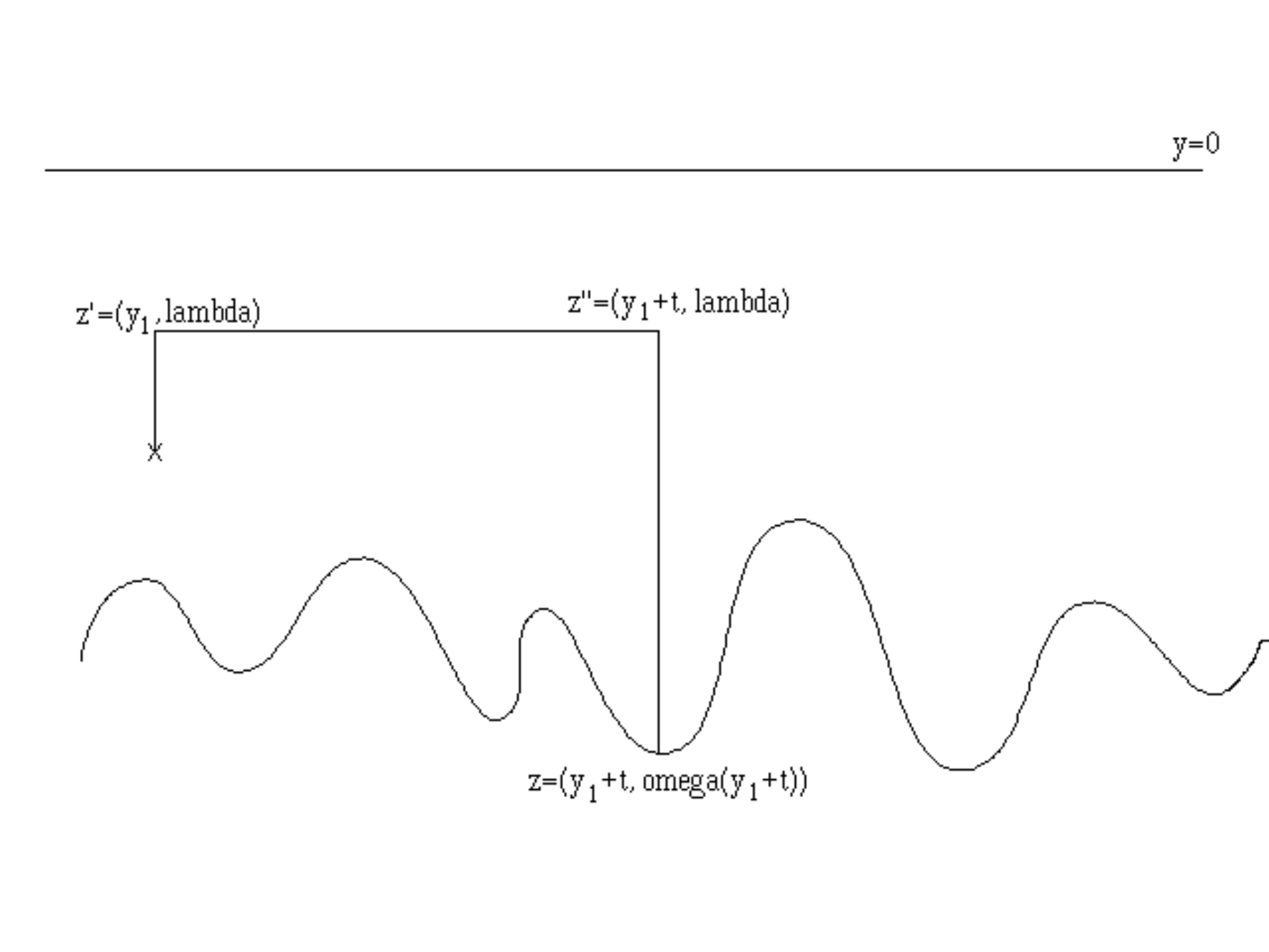}
\end{center}
\caption{The path $\ell_{y,t}$.}\label{fig:chemin}
\end{figure}

Notice that
$\ell_{y,t}$ depends in fact on $\lambda$, although the dependance is
omitted in order not to burden the notation. With this choice, we have 
\begin{eqnarray*}
\left|u(y)\cdot \nu (y_1 + t)  \right|&\leq&
\int_{[y_2,\lambda]} |\pa_2u|(y_1, y_2')\:dy_2'\\ 
&&+ \int_0^t |\pa_1 u|(y_1+y_1', \lambda)\:dy_1'\\
&&+\int_{\omega(y_1+t)}^\lambda |\pa_2 u|(y_1+t, y_2')\:dy_2'	\\
&\leq & \int_{\omega(y_1)}^0|\pa_2u|(y_1, y_2')\:dy_2' + \int_0^B
|\pa_1 u|(y_1+ y_1', \lambda)\:dy_1' \\&&+
\int_{\omega(y_1+t)}^0|\pa_2u|(y_1+t, y_2')\:dy_2'. 
\end{eqnarray*}
Integrating with respect to $y$ and $t$, we obtain, for all
$\lambda\in(\sup \omega, 0)$ 
$$
\int_{y\in R}\int_0^B\left|u(y)\cdot \nu(y_1 + t)  \right|^2\:dy\:dt
\leq  C_B\left( \int_R |\pa_2 u|^2(y)\:dy + \int_{\R}|\pa_1 u|^2(y_1,
\lambda)\:dy_1\right). 
$$ 
Integrating once again with respect to $\lambda$ yields the desired inequality.
\end{proof}

$\bullet$ Let us now examine in which case assumption
\eqref{hyp:nondegenerate} is satisfied in the periodic, quasi-periodic
and stationary ergodic settings: first, if $\omega$ is $\T$-periodic,
where $\T:=\R/\Z$, then \eqref{hyp:nondegenerate} merely amounts to 
$$ 
\int_{\T}{\omega'}^2>0.
$$
Hence (H') holds as soon as the lower boundary is not flat. In this
case assumption \eqref{hyp:nondegenerate} is necessary, as shows the
following example: assume that  
$\omega\equiv -1$, and consider the sequence $(u_k)_{k\geq 1}$ in
$H^1(R)$ defined by $u_{k,2}\equiv 0$ and 
$$
u_{k,1}(y)=\left\{ \begin{array}{ll}
                   	1&\text{ if }|y_1|\leq k,\\
			0&\text{ if }|y_1|\geq k+1,
                   \end{array}
\right.
$$
and $u_{k,1}(y)\in[0,1]$ for all $y$, $\|\nabla u_{k,1}\|_{L^\infty}\leq 2.$

Then it is easily checked that $u_k\cdot \nu\vert_\Gamma=0,$ and that 
$$
\|u_k\|_{L^2(R)}\geq 2 k.
$$
On the other hand,
$$
\|\nabla u_k\|_{L^2(R)}^2=\int_{k\leq y_1\leq k+1}|\nabla u_k|^2 \leq
8\quad\forall k\geq 1. 
$$
Hence assumption (H') cannot hold in $R$.

In the quasi-periodic case, the situation is similar to the one of the
periodic case, i.e. 
$$
\eqref{hyp:nondegenerate}\iff \omega'\neq 0.
$$
Indeed, assume that 
$$
\omega(y_1)=F(\lambda y_1)
$$
for some $\lambda\in\R^d$, $F\in\mathcal C^2(\T^d)$, with $d\geq 2$
arbitrary. Then  
$$
\int_0^A {\omega'}^2(y_1+t)\: dt = \int_0^A (\lambda\cdot \nabla
F)^2(\lambda(y_1+t))\:dt. 
$$
Write $F$ as a Fourier series:
$$
F(Y)=\sum_{k\in\Z^d}a_k e^{2i\pi k\cdot Y}\quad\forall Y\in\T^d.
$$
Then
\begin{eqnarray*}
\int_0^A (\lambda\cdot \nabla F)^2(\lambda(y_1+t))\:dt&=&-2\pi
\sum_{\substack{k,l\in\Z^d,\\\lambda\cdot(k+l)\neq 0}} a_ka_l \:(
\lambda\cdot k)\:	(\lambda\cdot l) e^{2i\pi(k+l)\cdot \lambda
  y_1}\frac{e^{2i\pi (k+l)\cdot \lambda A} -1}{i(k+l)\cdot \lambda}\\ 
&&+4\pi^2 A \sum_{\substack{k,l\in\Z^d,\\\lambda\cdot(k+l)= 0}}
a_ka_l(\lambda\cdot k)^2. 
\end{eqnarray*}The first term is bounded uniformly in $y_1$ and $ A$ provided the
sequence $a_k$ is sufficiently convergent and $\lambda$ satisfies a diophantine condition.
Consequently, setting $$C_0=4\pi^2\sum_{\substack{k,l\in\Z^d,\\\lambda\cdot(k+l)= 0}}
a_ka_l(\lambda\cdot k)^2,$$
we deduce that there exists a constant $C$ such that
$$
\forall A>0,\ \forall y_1\in\R,\quad C_0A-C\leq \int_0^A {\omega'}^2(y_1+t)\: dt\leq C_0A+C.
$$
The above inequality entails that $C_0\geq 0$. If $C_0>0,$ inequality \eqref{hyp:nondegenerate} is proved. If $C_0=0$, we infer that
$$
\int_\R{\omega'}^2<\infty.
$$
As a consequence, since $\omega'$ is uniformly continuous on $\R$, $\lim_{|t|\to\infty}\omega'(t)=0$. On the other hand, it can be proved thanks to classical arguments that for all $\eps>0, N>0$, there exists $n\in\N$ such that $n>N$ and
$$
d(\lambda n, \Z^d)\leq \eps.
$$
For $\eps$ small and $N$ large, and $t$ in a fixed and arbitrary bounded set, we obtain
\begin{eqnarray*}
\om'(t+n)&=&o(1)\\
&=&\lambda\cdot\na F(\lambda t + \lambda n)\\
&=&	\lambda\cdot\na F(\lambda t) + o(1)\\
&=&\om'(t)+ o(1).
\end{eqnarray*}
Thus $\om'(t)=0$, and $\om'\equiv 0$.

Hence we deduce that \eqref{hyp:nondegenerate} is satisfied as soon as
$\omega'$ is not identically zero, at least for ``generic'' quasi-periodic functions (i.e. such that the Fourier coefficients of the underlying periodic function are sufficiently convergent and such that $\lambda$ satisfies a diophantine condition). In fact, slightly more refined arguments (which we leave to the reader) show that the result remains true as long as 
$$
\sum_{k\in\Z^d} |k|\; |a_k| <\infty,
$$
without any assumption on $\lambda$.

\vskip1mm

Let us now give give another formulation of \eqref{hyp:nondegenerate}
in the stationary ergodic case. We denote by $(M,\mu)$ the underlying
probability space, and by $(\tau_{y_1})_{y_1\in\R}$ the
measure-preserving transformation group acting on $M$. We recall that
there exists a function $F\in L^\infty(M)$ such that 
$$
\omega(y_1,m)=F(\tau_{y_1} m),\quad y_1\in R,\ m\in M.
$$
As in \cite{Basson:2007}, we define the stochastic derivative of $F$ by
$$
\pa_m F(m):=\omega'(0, m)\quad\forall m\in M,
$$
so that $\omega'(y_1,m)=\pa_m F(\tau_{y_1}m)$ for $(y_1,m)\in \R\times
M.$ We claim that almost surely in $m\in M$, 
$$
\inf_{y_1\in \R}\int_0^A |\omega'(y_1+ t, m)|^2\:dt = \essinf_{m'\in
  M} \int_0^A |\pa_m F(\tau_t m')|^2 \:dt.  
$$
Indeed, notice that the left-hand side is invariant under the
transformation group $(\tau_{z_1})_{z_1\in\R}$ as a function of $m\in
M$. As a consequence, it is constant (almost surely) over $M$; we
denote by $\phi$ the value of the constant. Since $\omega'\in
L^\infty$, we also have 
$$
\phi=\inf_{y_1\in \Q}\int_0^A |\omega'(y_1+ t, m)|^2\:dt \quad \text{a.s. in }M.
$$
Now, for all $y_1\in \Q$, 
\begin{eqnarray*}
\int_0^A |\omega'(y_1+t, m)|^2\:dt&=&\int_0^A |\pa_m F(\tau_t (\tau_{y_1} m))\:dt\\
&\geq & \essinf	_{m'\in M} \int_0^A |\pa_m F(\tau_t m')|^2 \:dt=:\phi'
\end{eqnarray*}
almost surely in $M$. Taking the infimum over $y_1\in \Q,$ we infer that $\phi\geq \phi'.$

On the other hand, by definition of $\phi'$, for all $\eps>0$, there
exists $\mathcal M_\eps\subset M$ such that $P(\mathcal M_\eps)>0$ and 
$$
\phi'\leq \int_0^A |\pa_m F(\tau_t m)|^2 \:dt\leq \phi' +
\eps\quad\forall m\in \mathcal M_\eps. 
$$
Consequently, for all $m\in\mathcal M_\eps$, we have
$$
\inf_{y_1\in\R} \int_0^A |\omega'(y_1+t, m)|^2\:dt \leq  \int_0^A
|\omega'(t, m)|^2\:dt\leq \phi'+ \eps, 
$$
that is,
$$
\phi \leq \phi'+\eps.
$$
Hence $\phi=\phi'.$ Eventually, we deduce that in the stationary
ergodic case, assumption \eqref{hyp:nondegenerate} is equivalent to  
\be\label{hyp:ndg_ergodic}
\exists A>0,\ \essinf_{m\in M} \int_0^A |\pa_m F(\tau_t m)|^2 \:dt >0.
\ee
A straightforward application of the stationary ergodic theorem shows
that \eqref{hyp:ndg_ergodic} implies that 
$$
E[|\pa_m F|^2]>0.
$$
However, assumption \eqref{hyp:ndg_ergodic} appears to be much more
stringent than the latter condition: indeed, \eqref{hyp:ndg_ergodic}
is a uniform condition over the probability space $M$, whereas the convergence 
$$
\frac{1}{R}\int_0^R |\pa_m F(\tau_t m)|^2 \:dt
\underset{R\to\infty}{\longrightarrow} E[|\pa_m F|^2] 
$$
only holds pointwise.

\vskip1mm

$\bullet$ Let us now compare condition \eqref{hyp:nondegenerate} with
the assumption (H) of \cite{Bucur:2008}. In order to have a common
ground for the comparison, we assume that the setting is stationary
ergodic. In this case, the family of Young measures associated with
the sequence $\omega'(\cdot)/\eps$ can be easily identified: indeed,
according to the results of Bourgeat, Mikelic and Wright (see
\cite{Bourgeat:1994}), for all $G\in\mathcal C^1(\R)$ and for all test
function $\varphi \in L^1(\R\times M)$, there holds 
$$  
\int_{\R\times M} G\left(\omega'\left(\frac{y_1}{\eps},
m\right)\right) \varphi(y_1, m)\: d y_1 \: d\mu(m)\to \int_{\R\times
  M} E[G(\pa_m F)] \varphi(y_1, m)\: d y_1 \: d\mu(m). 
$$
By definition of the Young measure, the left-hand side also converges
(up to a subsequence) towards 
$$
\int_{\R\times M} \langle G, d\mu_{y_1}\rangle \varphi(y_1,m))\: d y_1 \: d\mu(m).
$$
As a consequence, we obtain
$$
\langle G, d\mu_{y_1}\rangle =  E[G(\pa_m F)] \quad\text{for a.e. }y_1\in\R.
$$
Hence condition (H) is equivalent (in the stationary ergodic setting) to 
$$
E[|\pa_m F|^2]>0,\quad\text{i.e. } F \text{ non constant a.s.}
$$
We deduce that assumptions (H) and (H') are equivalent in the periodic
and quasi-periodic settings. In the general stationary ergodic
setting, however, condition \eqref{hyp:nondegenerate} is stronger than
(H). But since we do not know whether \eqref{hyp:ndg_ergodic} is a
necessary condition for (H') in the stationary setting, we cannot
really assert that (H') is stronger than (H).

\subsection{Korn-type inequalities: assumption (H'')}

We now give a sufficient condition for (H''). Notice that our work in
this regard is related to the paper by Desvillettes and Villani
\cite{Desvillettes:2002}, in which the authors prove that for all
bounded domains $\Omega\subset \R^N$ which lack an axis of symmetry,
there exists a constant $K(\Omega)>0$ such that 
$$
\begin{aligned}
\|D(u)\|_{L^2(\Omega)}\geq K(\Omega) \|\nabla u \|_{L^2(\Omega)}\\
\forall u \in H^1(\Omega)^N\ \text{s.t. } u\cdot \nu\vert_{\pa\Omega}=0. 
\end{aligned}
$$
The differences with our work are two-fold: first, in our case, the
domain $R$ is an unbounded strip, which prevents us from using Rellich
compactness results in order to prove (H''). Moreover, the tangency
condition only holds on the lower boundary of $R$. However, as in
\cite{Desvillettes:2002}, we show that condition (H'') is in fact
related to the absence of rotational invariance of the boundary
$\Gamma$. Let us stress that this notion is related, although not
equivalent, to the non-degeneracy assumption of the previous paragraph
(see \eqref{hyp:nondegenerate}).

We first define the set of rotational invariant curves:

\begin{definition} For $(y_0,R)\in\R^2\times [0,\infty)$, denote by $\mathcal C(y_0,R)$ the circle with center $y_0$ and radius $R$.

For all $A>0$, we set
\begin{eqnarray*}
 \mathcal R_A&:=&\{ \gamma \in W^{1,\infty}([0,A])^2, \ \exists (y_0,
 R)\in \R^2\times (0,\infty),\ \gamma([0, A])\subset \mathcal C(y_0,
 R) \}\\ &&\cup \{ \gamma \in W^{1,\infty}([0,A])^2,
 \nu_\gamma=\text{cst.}\}, 
\end{eqnarray*} 
where $\nu_\gamma$ is a normal vector to the curve $\gamma$, namely
$$\nu_\gamma =\frac{1}{({\gamma_1'}^2+{\gamma_2'}^2)}\begin{pmatrix}
  -\gamma_2'\\\gamma_1'        
                                          \end{pmatrix}. 
$$

\end{definition}
Notice that $\mathcal R_A$ is a closed set with respect to the weak -
$\ast$  topology in $W^{1,\infty}$.

We then have the following result:

\begin{lemma}
Let $\omega\in W^{1, \infty}(\R)$. For $A>0$, $k\in\Z$, let 
$$
\gamma_k^A:y_1\in [0,A]\mapsto (y_1, \omega(y_1+kA)).
$$

Assume that there exists $A>0$ such that
\be\label{hyp:Korn}
\overline{\{\gamma_k^A, k\in \Z\}}\cap  \mathcal R_A=\emptyset ,
\ee
where the closure is taken with respect to the weak - $ \ast$ topology
in $W^{1, \infty}.$ 
Then (H'') holds.
\label{lem:Korn}
\end{lemma}

Assumption \eqref{hyp:Korn} means that each slice of length $A$ of the
boundary remains bounded away from the set of curves which are
invariant by rotation.  In particular, in the periodic case, a simple
convexity argument shows that all non-flat boundaries satisfy
\eqref{hyp:Korn} (it suffices to choose $A$ equal to the period of the
function $\omega$).

The proof of Lemma \ref{lem:Korn} uses the following technical result:
\begin{lemma}
For all $Y>\sup \omega$, let 
$$
R_Y:=\{ y \in \R^2, \; \omega(y_1)<y_2<Y\}.
$$ 

Consider the assertion
$$
(K_Y)\quad \exists C_Y>0, \ \forall u\in H^1(R_Y)\text{ s.t. } u\cdot
\nu \vert_\Gamma=0,\ \int_{R_Y} |u(y)|^2\; dy \leq C_Y \int_{R_Y}
|D(u)|^2. 
$$

If there exists $Y_0$ such that $(K_{Y_0})$ is true, then $(K_Y)$ is
true for all $Y>\sup \omega$. 

\label{lem:Korn_eq}
\end{lemma}
We postpone the proof of Lemma \ref{lem:Korn_eq} until the end of the section.

Let us now prove Lemma \ref{lem:Korn}: the idea is to reduce the
problem to the study of a Korn-like inequality in a fixed compact set,
and then to use standard techniques similar to the proof of the
Poincar\'e inequality in a bounded domain.  
\vskip1mm

\noindent \textit{First step: reduction to a flat strip.}

According to Lemma \ref{lem:Korn_eq}, it is sufficient to prove the
result in a domain $R_Y$ for some $Y>\sup \om$ sufficiently large (notice that
the boundary $\Gamma$ is common to all domains $R_Y$). 

We use the extension operator for Lipschitz domains defined by Nitsche
in \cite{Nitsche:1981}. Since the result of Nitsche is set in a
half-space over a Lipschitz curve, we recall the main ideas of the
construction, and show that all arguments remain valid in the case of
a strip, provided the width of the strip is large enough. 

We denote by $\Omega_-$ the lower half-plane below $\Gamma$, namely
$$
\Omega_-:=\{ y\in \R^2,\ y_2<\omega(y_1)\}.
$$

According to the results of Stein (see \cite{Stein:1970}), there
exists a ``generalized distance'' $\delta\in \mathcal C^\infty
(\Omega_-)$ such that 
$$
\begin{aligned}
0< 2(\omega(y_1)-y_2) \leq \delta(y)\leq
C_0(\omega(y_1)-y_2)\quad\forall y\in \Omega_-,\\ 
|\pa^\alpha_y \delta(y) |\leq C_\alpha\delta(y)^{1-|\alpha|}\quad
\forall \alpha \in \N^2\ \forall y\in \Omega_-. 
\end{aligned}
$$
(In general, since $\omega$ is merely a Lipschitz function, the
function $d(\cdot, \Gamma)$ has very little regularity, whence the
need for a generalized distance.)

Let $\psi\in \mathcal C([1,2])$ such that
$$
\int_1^2 \psi(\lambda)\:d\lambda=1,\quad \int_1^2 \lambda \psi(\lambda)\:d\lambda=0.
$$
For $u \in H^1(R_Y)$,  define an extension $\tilde u$ of $u$ in a strip
$(\inf \omega - \eta, Y)$ for some $\eta>0$ by  
\begin{eqnarray*}
 \tilde u (y)&=&u(y)\quad \text{if }y\in R_Y,\\
\tilde u_i (y):&=&\int_1^2 \psi(\lambda) \left[ u_i(y_\lambda) +
  \lambda \pa_i \delta (y) \: u_2(y_\lambda) \right]\:d\lambda,\text{
  if }y\in\Omega_- \\ 
&&\text{where } y_\lambda:=(y_1, y_2 + \lambda \delta(y)),\\
\end{eqnarray*}

Choose $Y$  such that
$$
Y>2C_0\sup\om - (2C_0-1) \inf \om.
$$
Then if $\eta>0$ is sufficiently small, $y_\lambda\in R_Y$ for all
$y\in (\inf \omega - \eta, Y) $. The function $\tilde u$ thus defined
does not have any jump across $\Gamma$. Moreover, it can be checked
that 
$$
\| D(\tilde u)\|_{L^2(\R\times (\inf \omega - \eta, Y))} \leq C \| D(u)\|_{L^2(R)}.
$$
Indeed, if $y\in \Omega_-$ and $y_2>\inf \omega -\eta$,
\begin{eqnarray*}
\left[\pa_i \tilde u_j + \pa_j \tilde u_i\right](y)&=& \int_1^2
d\lambda \psi(\lambda) \left[ \left(\pa_i u_j + \pa_j
  u_i\right)(y_\lambda) + 2 \lambda^2 \pa_i\delta(y)  \pa_j\delta(y)
  \pa_2 u_2(y_\lambda)\right.\\ 
&&\qquad \qquad\qquad+ \lambda \pa_i \delta(y) \left(\pa_2 u_j + \pa_j
  u_2\right)(y_\lambda)\\ 
&&\qquad \qquad\qquad+ \lambda \pa_j \delta(y) \left(\pa_2 u_i + \pa_i
  u_2\right)(y_\lambda)\\ 
&&\qquad \qquad\qquad\left.+ 2 \lambda \pa^2_{ij} \delta(y) u_2(y_\lambda)\right].
\end{eqnarray*}
Writing $u_2(y_\lambda)$ as 
$$
u_2(y_\lambda)=u_2(y_1, y_2 + \delta(y))+ \int_1^\lambda \pa_2 u_2 (y_\mu)\:d\mu
$$
and using the condition $\int_1^2 \lambda \psi(\lambda)\:d\lambda=0$ yields
$$
|D(\tilde u)(y)|\leq C \int_1^2 |D( u)|(y_\lambda)\: d\lambda.
$$
A careful analysis of the right-hand side then allows to prove that
$$
\int_\R d y_1 \int_{\inf \omega-\eta}^{\omega(y_1)} dy_2 |D(\tilde
u)(y)|^2\leq C \int_{R_Y} |D(u)|^2. 
$$
For all additional details, we refer to \cite{Nitsche:1981}.

Consequently, we have built an extension operator 
$$
E: H^1(R_Y)\mapsto H^1(\R\times (\inf \omega -\eta, Y))
$$
such that for all $u\in H^1(R_Y)$,
$$\begin{aligned}
\|D(u)\|_{L^2(R_Y)}\leq \|D(Eu)\|_{L^2(\R\times (\inf \omega
  -\eta))}\leq C \|D(u)\|_{L^2(R_Y)},\\    
\|u\|_{L^2(R_Y)}\leq \|Eu\|_{L^2(\R\times (\inf \omega -\eta))}\leq C \|u\|_{L^2(R_Y)}.
  \end{aligned}
$$

\vskip2mm

\noindent \textit{Second step: compactification of the problem.}

In the rest of the proof, we set $Q:=\R\times (\inf \omega -\eta,
Y)$. According to the first step, we now have to prove the existence
of a constant $C$ such that for any function $u\in H^1(Q)$ satisfying
$u\cdot \nu\vert_\Gamma =0$, 
$$
\|u\|_{L^2(Q)}\leq C\|D(u)\|_{L^2(Q)}.
$$

Of course, it is sufficient to prove that there exists a constant
$C_A$ such that for all $k\in \Z$,  
\be\label{in:Korn_compact}
\|u\|_{L^2(Q_{k,A})}\leq C_A\|D(u)\|_{L^2(Q_{k,A})}\quad \forall u \in
H^1(Q_{k,A})\text{ s.t. } u\cdot \nu\vert_{\Gamma_k^A}=0 
\ee
where
$$
Q_{k,A}= Q\cap \{ y, \: kA < y_1< (k+1) A\}.
$$

Assume by contradiction that \eqref{in:Korn_compact} is false. Then
there exists a sequence of relative integers $(k_n)_{n\geq 1}$ and a
sequence $u_n \in H^1(Q_{k_n, A})$, such that for all $n$, 
$$
\|u_n\|_{L^2(Q_{k_n,A})}\geq n \|D(u_n) \|_{L^2(Q_{k_n,A})}.
$$
In the rest of the proof, we drop all sub- and superscripts $A$ in
$Q_{k,A}$, $\gamma_k^A$ in order to lighten the notation.

Let $v_n:=u_n( \cdot +(k_n, 0))/\|u_n\|_{L^2}$. Then $v_n \in
H^1(Q_{0})$ for all $n$ and  
$$
v_n\cdot \nu\vert_{\Gamma_{k_n}}=0,\quad \|v_n\|_{L^2(Q_{0})}=1,
\  \|D(v_n)\|_{L^2(Q_{0})}\leq \frac{1}{n}. 
$$
According to the standard Korn inequality (see for instance
\cite{Nitsche:1981}), there exists $C>0$ such that for all $v\in
H^1(Q_0)$,  
$$
\|\nabla v\|_{L^2(Q_{0})}\leq C (\| v\|_{L^2(Q_{0})}+\|D (v)\|_{L^2(Q_{0})}).
$$
As a consequence, the sequence $v_n$ is bounded in $H^1(Q_0)$. By
Rellich compactness, there exists a subsequence (still denoted by
$v_n$) and a limit function $\bar v\in H^1(Q_0)$ such that 
$$
\begin{aligned}
v_n \rightharpoonup \bar v \quad\text{in } w-H^1(Q_0),\\
v_n \to \bar v   \quad\text{in }L^2(Q_0).
\end{aligned}
$$
We deduce that $D(\bar v)=0$ and $\| \bar v\|_{L^2}=1$. Hence $\bar v$
is a non-zero solid vector field: there exists $(C, y_0)\in (\R\times
\R^2)\setminus\{ 0\}$ such that 
$$
\bar v(y)=(C y+ y_0)^\bot \quad\text{for a.e. }y\in Q_0.
$$
On the other hand, for all $n\in \N$, for almost every $y_1\in [0,A]$, we have 
\be\label{eq:vn}
v_{n, 1} (\gamma_{k_n}(y_1))\gamma_{k_n, 2} '(y_1)-v_{n, 2}
(\gamma_{k_n}(y_1))\ \gamma_{k_n, 1}'(y_1)=0. 
\ee
Since the sequence $ \gamma_{k_n}$ is bounded in $W^{1, \infty}$, up
to the extraction of a further subsequence, $\gamma_{k_n}$ converges
weakly - $\ast$ in  $W^{1, \infty}$ towards a function $\bar
\gamma$. Since $\gamma_{k_n,1}(y_1)=y_1$ for all $n$, we deduce that
$(\gamma_{k_n}-\bar \gamma)\cdot e_1=0$. We then pass to the limit in
the identity \eqref{eq:vn} using the following facts: 

\begin{itemize}
 \item $\gamma_{k_n}\to \bar \gamma$ in $L^\infty$, and thus 
$$
\int_0^A \left| v_n(\gamma_{k_n}) - v_n(\bar \gamma)\right|^2\leq C \|
\gamma_{k_n} -\bar \gamma \|_\infty \| \nabla v_n\|_{L^2}^2 \to 0; 
$$

\item $v_n(\bar \gamma)$ is bounded in $H^{1/2}((0,A))$, and thus 
$$
v_n(\bar \gamma)\to \bar v (\bar \gamma)\quad\text{in }L^2(0,A).
$$
\end{itemize}
At the limit, we obtain
$$
(C\bar \gamma +y_0)\cdot \bar \gamma' =0,
$$
i.e. 
$$
| C\bar \gamma+y_0|^2=\text{cst.}
$$
We deduce that $\bar \gamma\in \mathcal R_A$, and thus $\mathcal
R_A\cap \overline{\{ \gamma_k, k\in\Z\}}\neq \emptyset$, which
contradicts the assumption of the lemma. Thus \eqref{in:Korn_compact}
holds, which completes the proof.\qed 

\begin{remark}\begin{itemize}
               \item We emphasize that condition \eqref{hyp:Korn} is
                 probably not optimal. Indeed, \eqref{hyp:Korn}
                 amounts to requiring that the inequality  
$$
\|u\|_{L^2}\leq C \|D( u)\|_{L^2}
$$
holds uniformly in each slice of length $A$.
However, since our proof relies on compactness results in $L^2$, it
seems necessary to work in a fixed compact domain. Of course, if a
more ``constructive'' proof were at hand (in the spirit of Lemma
\ref{lem:Poincare}), it is likely that \eqref{hyp:Korn} could be
weakened.

\item We have already pointed out that in the periodic case,
  conditions \eqref{hyp:Korn} and \eqref{hyp:nondegenerate} are
  equivalent. In the general case, however, \eqref{hyp:Korn} is
  stronger than \eqref{hyp:nondegenerate}. Indeed,
  \eqref{hyp:nondegenerate} merely requires the frontier $\Gamma$ to
  be non-flat (uniformly on $R$), whereas \eqref{hyp:nondegenerate}
  requires that is not invariant by rotation, in addition to being
  non-flat.  

\item We have used in the proof the following Korn inequality: since
  the function $\omega$ is Lipschitz continuous, there exists a
  constant $C_K>0$ such that 
$$
\| u\|_{H^1(R)}\leq C_K(\|u \|_{L^2(R)}+ \|D (u) \|_{L^2(R)})\quad \forall u\in H^1(R).
$$
We refer to \cite{Nitsche:1981} (see also \cite{Duvaut}) for a
proof. The constant $C_K$ depends only on the Lipschitz constant of
$\omega$. The inequality holds without any assumption on the
non-degeneracy of the boundary or on the behaviour of $u$ at the
boundary $\Gamma$. 
\end{itemize}
\end{remark}

We now prove Lemma \ref{lem:Korn_eq}. Assume that there exists $Y_0$ such that $(K_{Y_0})$ holds true. Let us first prove that $(K_Y)$ is true for all $Y\in (\sup \omega, Y_0)$. Let $u\in H^1(R_Y)$ be arbitrary. Using a construction similar to the one  of Nitsche (see \cite{Nitsche:1981}), we define an extension $u_1\in H^1(R_{Y_1})$ of $u$ such that 
$$\begin{aligned}
Y_1=\sup \omega + 2 (Y-\sup \omega),\\
\|u_1\|_{L^2(   R_{Y_1})}\leq C_1\|u\|_{L^2(   R_{Y})},\quad\|D (u_1)\|_{L^2(   R_{Y_1})}\leq C_1\|D(u)\|_{L^2(   R_{Y})}.
  \end{aligned}
$$ 

Iterating this process, we define sequences $(Y_n)_{n\geq 0}, (u_n)_{n\geq 0}$ such that $u_n\in H^1(R_{Y_n})$ and $u_{n+1}$ is an extension of $u_n$ for all $n$, and
$$\begin{aligned}
Y_{n+1}=\sup \omega + 2 (Y_n-\sup \omega),\\
\|u_{n+1}\|_{L^2(   R_{Y_{n+1}})}\leq C_{n+1}\|u_n\|_{L^2(   R_{Y_n})},\quad\|D( u_{n+1})\|_{L^2(   R_{Y_{n+1}})}\leq C_{n+1}\|D(u_n)\|_{L^2(   R_{Y_n})}.
  \end{aligned}
$$ 
It can be easily checked that $\lim_{n\to \infty}Y_n=\infty$, and thus there exists $n_0>0$ such that $Y_{n_0}>Y_0$. By construction, $u_{n_0}\in H^1(R_{Y_0})$ and there exists a constant $C$ such that
$$
\|u_{n_0}\|_{L^2(   R_{Y_{n_0}})}\leq C\|u\|_{L^2(   R_{Y})},\quad\|D (u_{n_0})\|_{L^2(   R_{Y_{n_0}})}\leq C\|D(u)\|_{L^2(   R_{Y})}.
$$
Moreover,
$$
u=u_{n_0}\quad\text{on }R_Y.
$$
$\vartriangleright$ From $(K_{Y_0})$, we infer that
$$
\|u_{n_0}\|_{L^2(   R_{Y_{0}})}\leq C_{Y_0}\|D (u_{n_0})\|_{L^2(   R_{Y_{0}})},
$$
and thus
$$
\| u\|_{L^2(R_Y)}\leq C \| D(u)\|_{L^2(R_Y)}.
$$
Hence $(K_Y)$ is satisfied.

Let us now prove that $(K_Y)$ is also true for all $Y>Y_0$. Let $u\in H^1(R_Y)$ arbitrary; then $u\in H^1(R_{Y_0})$, and
$$
\| u\|_{L^2(R_{Y_0})}\leq C_{Y_0} \| D(u)\|_{L^2(R_{Y_0})}.
$$
Moreover, according to the classical Korn inequality in the channel $R_{Y_0}$, there exists a constant $C_K$ such that
$$
\| u\|_{H^1(R_{Y_0})}\leq C_K\left( \| u\|_{L^2(R_{Y_0})}+ \| D(u)\|_{L^2(R_{Y_0})}\right).
$$
Let $\Sigma:=\R\times \{Y_0\}$. Then
$$
\|u\|_{L^2(\Sigma)}\leq C \| u\|_{H^1(R_{Y_0})}\leq C \| D(u)\|_{L^2(R_{Y_0})}.
$$

Now, for any $y\in \R\times (Y_0, Y)$, let $y'\in\Sigma$ such that
$$
 y'=y +t(1,-1)\quad\text{for some } t\in\R.
$$
Then
$$
u(y)=u(y')+t\int_0^1(\pa_1-\pa_2)u(y+t(1-\tau)(1,-1))d\tau.
$$
Notice that
$$
(1,-1)\cdot(\pa_1-\pa_2)u=\pa_1 u_1 + \pa_2u_2 - (\pa_1u_2+\pa_2 u_1),$$
and thus
\begin{eqnarray*}
\int_\R\int^Y_{Y_0}\left|u(y)\cdot (1,-1)\right|^2\:dy&\leq& C\left(\|u\|_{L^2(\Sigma)}^2 + \int_\R\int^Y_{Y_0}\left|D (u)\right|^2\right) \\
&\leq& C \|Du\|^2_{L^2(R_Y)}.
\end{eqnarray*}
Similarly,
$$
\int_\R\int_Y^{Y_0}\left|u(y)\cdot (-1,-1)\right|^2\:dy\leq C \|D(u)\|^2_{L^2(R_Y)}.
$$

Eventually, we obtain
$$
\| u\|_{L^2(R_Y)}\leq C  \|Du\|^2_{L^2(R_Y)},
$$
which completes the proof.

\section{Estimates for the no-slip condition} \label{sec2}
In this section, we will prove Theorem \ref{Direstimates}. In other
words, we will establish,  for any sequence $\lambda^\eps \in [0,+\infty]$, the
  well-posedness result  and the error 
estimates that were  established in \cite{Basson:2007}  for  $\lambda^\eps = 0$. 
We will use the same general strategy, based on  the
work of Ladyzenskaya and Solonnikov \cite{Ladyvzenskaja:1983}. However, the handling of
the slip type conditions will require new arguments, due to a loss of
control on the skew-symmmetric part of the gradient. Moreover, we shall
specify the dependence of the error terms with respect to $\phi$. We shall of
course put the stress on these new arguments.  

\medskip
The starting point of the proof is an approximation scheme by
solutions in truncated channels. Therefore,we introduce the notations 
$$ \forall \, k,l \ge 0, \quad  U_{k,l} \: := \: U \cap \{ k < |x_1 | <
l\}, \quad U_k \: := \: U_{-k,k}, $$
for any set $U$ of $\R^2$.  We take as a  new unknown 
$$
 v \: :=  \: u^\eps - \tilde u^{0}, \quad \tilde u^0 \: := \: 
 1_{\Omega} \, u^0 $$
where $u^0$ is the
Poiseuille flow. As a new pressure, we take
$$ q \: := \: p \: + \: 12 \phi x_1. $$ 
 It  formally satisfies 
\begin{equation} \label{NSepsbis} 
\left\{
\begin{aligned}
 - \Delta v + \tilde u^0 \cdot \na v +  v \cdot \na \tilde u^0 + v \cdot \na
 v + \na q  &   \: = \:  1_{R^\eps} (-12 \phi, 0)  , \: x \in \Omega^\eps, \\
\div v  &   \: = \:  0, \: x \in \Omega^\eps, \\
 v\vert_{x_2 = 1}    \: = \:  0,  \: \int_{\sigma^\eps}v_1  &  \: = \:
 0, \\
v_\tau \vert_{\Gamma^\eps}  \:  =  \:  \lambda^\eps   &  ( D(v)\nu)_\tau\vert_{\Gamma^\eps},  \quad  v \cdot
\nu\vert_{\Gamma^\eps} \: 
= \:  0,
\end{aligned}
\right.
\end{equation}
where $\nu $ is the inward pointing normal vector on $\Gamma^\eps$ and
$$v_\tau=(I_d-\nu\otimes\nu)v$$
denotes the tangential part of $v$ on $\Gamma^\eps$. The system is supplemented with
 the following jump conditions at the interface $\Sigma$: 
$$ [v]\vert_{\Sigma} = 0, \quad [ -D(v) e_2 + q e_2]\vert_{\Sigma} =
(-6\phi,0). $$ 
In order to build and estimate the field $v$, we consider the
approximate problems in $\Omega^\eps_n$
\begin{equation} \label{NSepsn}  
\left\{
\begin{aligned}
 - \Delta v + \tilde u^0 \cdot \na v +  v \cdot \na \tilde u^0 + v \cdot \na
 v + \na q  &   \: = \:  1_{R^\eps} (12 \phi, 0)  , \: x \in \Omega^\eps_n, \\
\div v  &   \: = \:  0, \: x \in \Omega_n^\eps, \\
 v\vert_{x_1 =  n} \: =  \:  v\vert_{x_1 =-n} \: =  \: v\vert_{x_2 =
   1}     & \:  = \:  0, \\
v_\tau \vert_{\Gamma^\eps_n}  \:    =  \:  \lambda^\eps   &   (D(v)\nu)_\tau\vert_{\Gamma^\eps_n},   \quad   v
\cdot \nu\vert_{\Gamma^\eps_n} \: = \:  0, 
\end{aligned}
\right.
\end{equation}
and 
$$  [v]\vert_{\Sigma_n}  \: = \:  0, \quad [ -D(v) e_2  + q
    e_2]\vert_{\Sigma_n}  \: = \:  (-6\phi,0).$$
The proof divides into four steps: 
\begin{enumerate}
\item  We construct  a solution $v_n \: $ of the approximate system \eqref{NSepsn}.
\item  We derive  $H^1_{uloc}$ estimates on $v_n$ that are uniform
in $n$. This yields compactness of $(v_n)_n$, hence, as $n$ goes to
infinity, a solution $u$ of \eqref{NSeps}. 
\item  We prove uniqueness of this solution $u$. 
\item We deduce from the previous steps the desired $O(\sqrt{\eps})$
  bound in $H^1_{uloc}$ for $u - u^0$. From there, using duality
  arguments, we get the $O(\eps)$ bound in  $L^2_{uloc}$. 
\end{enumerate}

\medskip
{\em Step 1.} The wellposedness of \eqref{NSepsn} relies on an {\it a
  priori} estimate over $\Omega^\eps_n$. Multiplying formally by $v$, we obtain
\begin{equation*}
 \int_{\Omega^\eps_n} |D(v)|^2 \: + \: (\lambda^\eps)^{-1} \int_{\Gamma^\eps_n}
|v_\tau|^2 \:  =  \: -\int_{\Omega_n^\eps} \left( \tilde u^0 \otimes v +
v \otimes \tilde u^0 \right) : D(v) - \int_{R^\eps_n} 12 \phi \,
v_1  + \int_{\Sigma^\eps_n} 6 \phi v_1, 
\end{equation*}
where $v_\tau$ denotes the tangential part of $v$. As $\|\na
u^0\|_\infty \: \le \:  C \, \phi $, we obtain  
\begin{equation}  \label{estimatevepsn} 
  \|D(v) \|_{L^2(\Omega^\eps_n)}^2 \: \le \:  C \, \phi \,  \left( \|v
\|_{L^2(\Omega^\eps_n)} \,  \|D(v) \|_{L^2(\Omega^\eps_n)} \: + \:
\sqrt{n \, \eps} \,  \| v \|_{L^2(R^\eps_n)} \: + \: \sqrt{n} \, \| v
\|_{L^2(\Sigma^\eps_n)}\right). 
\end{equation}
As $v$ is zero at the upper boundary of the channel, Poincar\'e
inequality applies, to provide
$$  \|v \|_{L^2(\Omega^\eps_n)} \: \le \: C \,  \| \na v
\|_{L^2(\Omega^\eps_n)}. $$ 
where $C$ depends only on the height of the channel. Let 
$$\Omega^\eps_{bl} \: := \: \{x , \: x_2 > \eps \omega(x_1/\eps)\}$$ 
the ``rough'' half plane, and $\tilde v \in H^1(\Omega^\eps_{bl})$
the extension of $v$ which is zero outside $\Omega^\eps_n$. We can apply
to $\tilde v$  the results of Nitsche \cite{Nitsche:1981} on the Korn inequality in a half
plane bounded  by a Lipschitz curve: one has
$$ \| \na \tilde v  \|_{L^2(\Omega^\eps_{bl})} \: \le \: {\cal C} \,
\| D(\tilde v)  \|_{L^2(\Omega^\eps_{bl})} $$ 
where the constant ${\cal C}$ only depends on the Lipschitz constant of the
curve.  For $\Omega^\eps_{bl}$, this Lipschitz constant, and therefore
the estimates, are  uniform in $\eps$. We insist that this inequality
is homogenenous: it does not involve the $L^2$ norm of $\tilde{v}$,
contrary to the more general inhomogeneous Korn inequality.   
Back to $v$, we get 
$$   \|\na v \|_{L^2(\Omega^\eps_n)} \: \le \: C \| D(v) \|_{L^2(\Omega^\eps_n)}. $$ 
Hence,  denoting for all $k \in \N$
$$ E_k \: :=  \:  \|v \|^2_{L^2(\Omega^\eps_k)}  \: + \: \|\na v
\|^2_{L^2(\Omega^\eps_k)} \: + \: \|D(v) \|^2_{L^2(\Omega^\eps_k)} $$
the combination of Poincar\'e and Korn inequalities  leads to $ E_n
\: \le \:  C \|D(v) \|^2_{L^2(\Omega^\eps_n)}$.  

\medskip
Now, by rescaling either the Poincar\'e inequality when $\lambda^\eps = 0$,
or the inequality in (H') when $\lambda^\eps \neq 0$, we get
$$ \| v \|_{L^2(R^\eps_n)} \: \le \:  C \,  \eps \, \| \na v
\|_{L^2(R^\eps_n)}. $$
Then, we deduce 
$$ \| v \|_{L^2(\Sigma^\eps_n)} \: \le \:  C \,  \sqrt{\eps} \, \| v
\|_{H^1(R^\eps_n)} \: \le \: C' \,  \sqrt{\eps}  \, \| \na v
\|_{L^2(R^\eps_n)}. $$ 
Back to the energy estimate \eqref{estimatevepsn}, we end up with 
\begin{equation*}
\begin{aligned}
 E_n \:  & \le  \:  C \, \| D(v) \|^2_{L^2(\Omega^\eps_n)} \: \le \: C'
 \, \phi \left( \|v
\|_{L^2(\Omega^\eps_n)} \,  \|D(v) \|_{L^2(\Omega^\eps_n)} \: + \:
\sqrt{n \, \eps} \,   \|\na v \|_{L^2(R^\eps_n)} \right) \\
& \le \: C' \phi \left(E_n \: + \: \sqrt{n \, \eps} \, 
\sqrt{E_n}\right) \: \le \:   C' \phi E_n \: + \: \frac{1}{2} E_n \: +
\:  \frac{C'^2}{2} \, \phi^2 \, n \, \eps.
\end{aligned}
\end{equation*}
so that, for $\phi$ small enough,  we have the global estimate
\begin{equation} \label{estimateEn}
 E_n \:  \le \: C \, \phi^2 \,  n \,  \eps.
\end{equation}

\medskip
Thanks to  this estimate, one obtains by classical arguments a
variational solution  $v_n \in 
H^1(\Omega^\eps_n)$ of \eqref{NSepsn}. The uniqueness of this solution
(for $\phi$ small enough) is deduced from the same kind of energy
estimates, performed on the difference of two solutions. We leave the
details to the reader. 

\bigskip
{\em Step 2.}
The next step in the proof of Theorem \ref{Direstimates} is the
derivation of uniform $H^1_{uloc}$ bounds on $v_n$. The idea, which
originates in a work of Ladyzenskaya and Solonnikov, {\em is to prove by
induction on $k'=n-k$} that 
\be\label{in:energy_rec}E_k \: \le \:  C_0 \, \phi^2 \, (k+1) \, \eps, \quad C \mbox{ large enough}.  \ee
 Once the bound on  the $E_k$'s  is proved, we can use it with
 $k=n-1$, so that   $E_1 \: \le  \: C \, \phi^2  \eps$.
 This gives a control on a unit slice of the channel
 around $x_1=0$. But as will be clear from the proof of the
 induction relation, $x_1 = 0$ plays no special role: in other words
 the same bound holds for any unit slice of the channel, which gives
 the uniform $H^1_{uloc}$ bound. 

\medskip
Let us now describe the induction process. First, by
\eqref{estimateEn}, the induction assumption holds 
with  $k'=0$. To go from $k'-1$ to $k'$, that is from $k+1$ to
$k$, we shall need the following inequality: 
\begin{equation} \label{ineqEk}
\forall k \le n,   \quad E_k \: \le \: C_1 \left( E_{k+1} - E_k \: +  \:
(E_{k+1} - E_k)^{3/2} \: + \: \phi^2 \, (k+1) \, \eps \right).
\end{equation}
This inequality at hand, and assuming $E_{k+1} \: \le \:  C_0 \, \phi^2
\, (k+2) \, \eps$, one obtains straightforwardly that  $E_k  \: \le \:
C_0\, \phi^2 \,  (k+1) \, \eps \:$ for all $k\geq C_0 -1$ provided $C_0$ is chosen large enough. For $k\leq C_0-1$, we have merely $E_k\le C_0\, \phi^2 \,  (\lfloor C_0 \rfloor+1) \, \eps \:$. Hence, up to a new definition of the constant $C_0$, we obtain inequality \eqref{in:energy_rec}.

\medskip
An inequality similar to \eqref{ineqEk} has been established by one of
the authors  in 
\cite{Basson:2007}, for the case $\lambda^\eps = 0$. The discrete variable $k$  is replaced
in \cite{Basson:2007} by a continuous variable $\eta$, but the correspondence from
one to another is obvious. We also refer to the original paper
\cite{Ladyvzenskaja:1983}, and to the boundary layer analysis in
\cite{GerMas}, in which similar inequalities are derived.

\medskip
 Relation \eqref{ineqEk} follows from localized energy
estimates. We introduce some truncation function $\chi_k = \chi_k(x_1)$, such that
$\chi_k = 1$ over $\Omega^\eps_k$,  $\chi_k=0$ outside
$\Omega^\eps_{k+1}$, and $|\chi'_k| \le 2$.   Multiplying by $\chi_k \, v$
within \eqref{NSepsn} and integrating by parts, we deduce that 
\begin{align*}
&\int_{\Omega^\eps} \chi_k |D(v) |^2 \: + \: (\lambda^\eps)^{-1}
\int_{\Gamma^\eps} \chi_k| v_\tau|^2  \\
 \: \le \: & \int_{\Omega^\eps} \chi_k
(\tilde u^0 \otimes v + v \otimes \tilde u^0) : D(v) \: - \: \int_{R^\eps}
12 \phi \chi_k v_1 \: + \:  \int_{\Sigma^\eps} 6 \phi v_1\chi_k \\
 &\: +\int_{\Omega^\eps} D(v) : ( \na \chi_k \otimes v) 
\:  + \:  \int_{\Omega^\eps}  (\tilde u^0 \otimes v + v \otimes \tilde u^0) :
(\na \chi_k \otimes v)  \\
& \:  + \: \int_{\Omega^\eps} (v \otimes v) :
(\na \chi_k \otimes v) \: + \:
\int_{\Omega^\eps} q\na \chi_k \cdot v   \: = \: \sum_{j=1}^7 I_j.  
\end{align*}
where $\na \chi_k = (\chi'_k, 0)$. 
The r.h.s. of the inequality has two different parts:
\begin{itemize}
\item The first three terms are very similar to those of step 1. They
  are treated along the same lines:
$$ \sum_{j=1}^3|I_j| \: \le \: C \, \phi \,
  \left(E_{k+1} + \sqrt{(k+1) \: \eps} \, \sqrt{E_{k+1}}\right). $$
\item The remaining terms involve derivatives of $\chi_k$: they are
  supported in $\Omega_{k,k+1}$. Standard manipulations yield the
  bounds:
$$  |I_4| \: + \: |I_5| \: \le \: C \,  (E_{k+1}- E_k), \quad |I_6|
  \: \le \: C (E_{k+1} - E_k)^{3/2}. $$
\item The treatment of the pressure term is a little more
  tricky. We  decompose 
$$ \Omega^\eps_{k,k+1} \: = \: \Omega_{k,k+1}^{\eps,-} \: \cup \: \Omega_{k,k+1}^{\eps,+},
  \quad \Omega_{k,k+1}^{\eps,\pm} \: :=  \: \Omega^\eps_{k,k+1} \cap \{\pm x_1 \ge
  0\}.  $$
The zero flux condition on $v$ implies that
$\int_{\Omega_{k,k+1}^{\eps,\pm}} f(x_1)  v_1 = 0$ for any function $f$
  depending only on $x_1$. Thus, 
$$ I_7 \: = \: \int_{\Omega^\eps_{k,k+1}} q \,  \chi'_k \, v_1 \: = \:
 \: = \:
\int_{\Omega^{\eps,-}_{k,k+1}} (q - q_k^-) \, \chi'_k \, v_1 \: + \:
\int_{\Omega^{\eps,+}_{k,k+1}} (q - q_k^+) \, \chi'_k v_1, $$
where $q_k^\pm$  is the average of $q$ over $\Omega_{k,k+1}^{\eps,\pm}$. Then,
we use the well-known estimate 
$$  \left\| q -  \oint_{{\cal O}} q\right\|_{L^2({\cal O})}  \: \le \: {\cal C} \, \|
\Delta v + f \|_{H^{-1}({\cal O})} $$
for the Stokes system 
$$-\Delta v + \na q = f, \quad \div u = 0 \mbox{ in } {\cal O}$$ 
where $C$ only depends on the measure of
${\cal O}$ and the Lipschitz constant of $\pa {\cal O}$. 
We take here  $\displaystyle {\cal O} =
\Omega_{k,k+1}^{\eps,\pm}$ (so that the constant is uniform in $k$ and $\eps$), and 
$$ f \: := \: -\div \left(\tilde u^0 \otimes v + v \otimes \tilde u^0 + v
\otimes v \right) + 1_{R^\eps} (12\phi,0).
$$

From there, one gets after a few computations
\begin{align*}
 | I_7 | \: & \le \:  C \, \left(  \| q -  q_k^-
 \|_{L^2(\Omega^{\eps,-}_{k,k+1})} + \| q -  q_k^+
 \|_{L^2(\Omega^{\eps,+}_{k,k+1})} \right)\,  \| v \|_{L^2(\Omega_{k,k+1})} \\
&  \le \:  C \,  \left( \phi \, \sqrt{\, \eps} \, \sqrt{E_{k+1} - E_k}  \: +
 \: (E_{k+1} - E_k) \: + \:  (E_{k+1} - E_k)^{3/2}  \right) \\
& \le \:  C \,  \left( \phi^2 \, \eps \: + \:  (E_{k+1} - E_k) \: + \:
 (E_{k+1} - E_k)^{3/2} \right). 
\end{align*}
\end{itemize}
We refer to \cite{Basson:2007} for more details. By gathering all the inequalitites
on the $I_j$'s, we obtain for $\phi$ small enough
\begin{eqnarray} \label{ineqEk2} 
\int_{\Omega^\eps} \chi_k |D(v) |^2   & \le & C \left( E_{k+1} - E_k \: +  \:
(E_{k+1} - E_k)^{3/2} \: + \: \phi^2 \, (k+1) \, \eps \right)\\
&&+ C \phi\,\left(E_k + \phi \sqrt{(k+1)\eps}\sqrt{E_k}\right)
\end{eqnarray}
Now, we have
\begin{align*}
\int_{\Omega^\eps} \chi_k |D(v) |^2  \: & \ge  \: \int_{\Omega^\eps}
\chi_k^2 |D(v) |^2 \: \ge \: \int_{\Omega^\eps}  |D(\chi_k v) |^2 \: -
\:  \int_{\Omega^\eps}   | \chi'_k |^2  |v^2| \\
\: & \ge  \:   \int_{\Omega^\eps}  |D(\chi_k v) |^2   - 4(E_{k+1} -
E_k). 
\end{align*}
As  $\chi_k \, v$ is zero outside $\Omega^\eps_{k+1}$, we can proceed
as in Step 1, to  get
$$ \int_{\Omega^\eps}  |D(\chi_k v) |^2 \: \ge \:  c \, E_{k+1} \: \ge
\:  c \, E_k + c \, (E_{k+1} - E_k). $$
Finally, 
$$ \int_{\Omega^\eps} \chi_k |D(v) |^2 \: \ge \:C\, E_k \: - \:  c \,
(E_{k+1} - E_k). $$
Combining this inequality with \eqref{ineqEk2} gives the result provided $\phi$ is small enough.

\medskip
As we have already explained, once inequality \eqref{ineqEk} is
proved, one obtains easily a $O(\sqrt{\eps})$ $\, H^1_{uloc}$ bound on
$v_n$. By standard compactness arguments, any accumulation point $v$
of $(v_n)$ is a  solution  of \eqref{NSepsbis}. It provides a
solution $u^\eps$ of the original system \eqref{NSeps}. Moreover,  
\begin{equation} \label{estimueps}
\| u^\eps - u^0\|_{H^1_{uloc}(\Omega^\eps)} \: \le \: C \, \phi \,
\sqrt{\eps}, \quad \| u^\eps \|_{H^1_{uloc}(\Omega^\eps)} \: \le \: C \,\phi. 
\end{equation}
 
\bigskip 
{\em Step 3.} It remains to prove the uniqueness of the solution
 in  $H^1_{uloc}$.

\medskip
Let now   $v = u^{\eps,2} - u^{\eps,1}$ the difference between two
solutions of \eqref{NSeps}. It satisfies 
$$ -\Delta v + \div \left( u^{\eps,1} \otimes  v +  v \otimes  u^{\eps,1}  + v
\otimes  v\right) + \na q  = 0 $$
together with $\div v = 0$ and homogeneous jump and boundary
conditions. We can always  assume that $u^{\eps,1}$
satisfies the bounds in \eqref{estimueps}.  Then, performing energy
estimates similar to those of  Step 2, we get, for $\phi$ small enough: 
\begin{equation*}
\forall k, \:  \quad E_k \: \le \: C \left( (E_{k+1} - E_k) \:  + \:  (E_{k+1}
- E_k)^{3/2} + \: \phi \, \sqrt{\eps} \, E_k\right)
\end{equation*} 
 We have used implicitly that
\begin{align*}
\left|\int_{\Omega^\eps} \chi_k u^{\eps,1} \otimes \, v : \na(v) \right|
\: & \le   \:  \sum_{j=0}^k \int_{\Omega^\eps_{j,j+1}} |u^{\eps,1}| \,
|v| \, | \na(v)| \\
& \le   \: \sum_{j=0}^k  \| u^{\eps,1} \|_{L^4(\Omega^\eps_{j,j+1})} \,
 \| v\|_{L^4(\Omega^\eps_{j,j+1})} \| \na v\|_{L^2(\Omega^\eps_{j,j+1})}
 \\
& \le  \:  \| u^{\eps,1} \|_{H^1_{uloc}} \sum_{j=0}^k \|v
 \|^2_{H^1(\Omega^\eps_{j,j+1})} \: \le \: C \phi \, \sqrt{\eps} E_{k+1}.
\end{align*} 
As $v$ belongs to $H^1_{uloc}$, $\: E_{k+1} - E_k $ is bounded uniformly
in $k$: eventually, for  $\phi \,  \sqrt{\eps}$  small enough, we get 
$ E_k \le C $ for all $k$, which means that $v$ is of finite
energy. The fact that $v= 0$ then follows from a classical {\em global energy
  estimate}, performed on the whole channel $\Omega^\eps$. This
concludes the proof. 

\bigskip
 {\em Step 4}. Note that, by the previous steps, we have  established
 not only the well-posedness, but the $H^1_{uloc}$ estimate
$$ \| u^\eps - u^0 \|_{H^1_{uloc}(\Omega)} \: \le \: C \,  \phi
 \sqrt{\eps}. $$

From there, one obtains that 
$$ \| u^\eps - u^0 \|_{L^2_{uloc}(\Sigma)} \: \le \: C \,  \phi \,
\eps.  $$
The $L^2_{uloc}(\Omega)$ estimate  follows from estimates on a 
 {\em linear problem in the channel $\Omega$}:
\begin{align*}
-\Delta v + \na q + u^0 \cdot \na v + v \cdot \na u^0 & = \div F^\eps, \\
\div v & = 0, \\
v\vert_{\Sigma} & = \varphi^\eps, \\
\end{align*}
where $v \, = \,  u^\eps - u^0$, $\, \varphi^\eps \, =  \,
v\vert_\Sigma \, = \, O(\phi \eps)$  in
$L^2_{uloc}$ and $\, F^\eps \, = \,  v \otimes  v \, =  \, O(\phi \eps)$ in
$L^2_{uloc}$   thanks to the $H^1_{uloc}$ bound. By a duality
argument, one can then prove that $\| v \|_{L^2_{uloc}(\Omega)}$ is also
$O(\phi \, \eps)$ in $L^2_{uloc}$.
This duality argument is explained in the paper \cite[section 3.2]{Basson:2007}:  as the
slip condition in  the boundary condition at $\pa \Omega^\eps$ plays
no role, we skip the proof.

\section{Boundary layer analysis} \label{sec3}
In order to improve our description of $u^\eps$, we must
analyze  the behaviour of the fluid in the boundary layer. 
The starting point of this analysis is a formal expansion: we
anticipate  that, near the rough boundary, we have
$$ u^\eps(x) \: = \:  u^0(x) \: + \: 6 \, \phi \,  \eps v(x/\eps) $$
where $u^0$ is the Poiseuille flow, and $v= v(y)$ is a boundary
layer corrector, due to the fact that $u^0$ does not satisfy either
the Dirichlet boundary condition when $\lambda = 0$, or the slip
boundary condition when $\lambda \neq 0$. We shall focus on the latter
case, which is the new one. Classically, the rescaled variable $y$
belongs to the bumped  half plane $\Omega^{bl} := \{y, \, y_2 >
\omega(y_1) \}$, and by plugging the expansion in \eqref{NSeps}, one finds that  
\begin{equation} \label{BL} \tag{BL}
\left\{
\begin{aligned}
-\Delta v + \na p  = 0, \quad & y \in \Omega^{bl}, \\ 
\div v  = 0, \quad & y \in \Omega^{bl}, \\ 
 (D(v) \nu)_\tau   = - (D((y_2, 0)) \nu)_\tau, \quad & y \in \pa \Omega^{bl}, \\
v \cdot \nu = - (y_2, 0) \cdot \nu,  \quad &  y \in \pa \Omega^{bl}.
\end{aligned}
\right.
\end{equation}
where 
$$\nu = \nu(y) \: := \: \frac{1}{\sqrt{1 + \gamma'^2(y_1)}}
(-\omega'(y_1),1) $$ 
is a unit normal vector. The inhomogeneous boundary terms
come from the Poiseuille flow ($x_2(1 - x_2) \approx \eps y_2$ near the boundary). 

\medskip
System \eqref{BL} is different  from the boundary layer system
met in the former studies on wall laws: it has inhomogeneous Navier 
boundary conditions, instead of Dirichlet ones. Nevertheless, we are
able to obtain similar results,  as regards well-posedness and
 qualitative issues. 
\subsection{Well-posedness of the boundary layer}
First, we have the following  well-posedness result: 
\begin{theorem}
System \eqref{BL} has a unique solution $v \in
H^1_{loc}(\overline{\Omega^{bl}})$ satisfying:
$$ \sup_{k} \int_{\Omega^{bl}_{k,k+1}} |
\na v |^2  < 
+\infty  \quad \mbox{ where for all }  \:  k,l, \quad
\Omega_{k,l}^{bl} := \Omega_{bl} \cap \{ k < y_1 < l \}.$$
\end{theorem}
The proof  follows closely the lines of \cite{GerMas}, where the
case of an inhomogenenous Dirichlet  condition  (instead of Navier)
was considered. The main difficulty comes from the unboundedness of
the domain, which prevents from using the Poincar\'e inequality or 
assumptions like (H') or (H''). To overcome this difficulty, there are
two main steps:
\begin{enumerate}
\item One replaces system \eqref{BL} by an equivalent system, {\em set
  in the channel} 
$$\Omega^{bl,-} \:  := \: \Omega^{bl} \cap \{y_2 < 0\}.$$
 This equivalent system involves a nonlocal
boundary condition at $y_2 = 0$, with  a {\em Dirichlet-to-Neumann} type
operator. 
\item Once brought back to the channel $\Omega^{bl,-}$, one can follow
  the same general strategy as in the previous section, based on the truncated
  energies
$$ E_k \: :=  \:  \|v \|^2_{L^2(\Omega^{bl,-}_k)}  \: + \: \|\na v
\|^2_{L^2(\Omega^{bl,-}_k)} \: + \: \|D(v) \|^2_{L^2(\Omega^{bl,-}_k)}. $$ 
\end{enumerate}
Let us give a few hints on these two steps.

\medskip
{\em Step 1.} It relies on  the notion  of {\em transparent boundary
conditions in numerical analysis}. The formal idea is the
following: the solution $v$ of \eqref{BL} satisfies the boundary value problem
\begin{equation} \label{Stokes} 
\left\{
\begin{aligned}
-\Delta  v + \na q & = 0, \: y_2 > 0, \\
\na \cdot v & = 0, \:  y_2 > 0, \\
v\vert_{y_2=0} & = v_0, 
\end{aligned}
\right.
\end{equation}
where $v_0 = v\vert_{y_2 = 0}$. Using the Poisson kernel for the
Stokes problem in a half-plane, we have the representation
formula: 
 \begin{equation} \label{poisson}
 v(y) \: = \: \int_{\R} G(t,y_2) v_0(y_1-t)\, dt, \quad q(y) \: = \:
 \int_\R \na g(t,y_2) \cdot v_0(y_1-t) \, dt 
\end{equation}
$$ \mbox{ where } \quad 
 G(y) \: = \: \frac{2y_2}{\pi(y_1^2 +y_2^2)^2} \left( \begin{smallmatrix}
     y_1^2 & y_1 \, y_2 \\ y_1 \, y_2 & y_2^2 \end{smallmatrix} \right),
 \quad g(y) = -\frac{2y_2}{\pi(y_1^2 +y_2^2)}.  $$
 Thanks to this representation formula, we can express the stress
 $$(2 D(v)\nu - q \nu)\vert_{y_2=0} = -2 \pa_2 v + (\pa_2 v_1 -
 \pa_1 v_2) e_1 + q e_2$$
 in terms of $v$ at $y_2 = 0$. Formally, this leads to some relation 
$$ (-2D(v)e_2 + q \,e_2)\vert_{\{y_2=0\}}  \: = \: DN(v\vert_{\{y_2=0\}}) $$  
for some Dirichlet-to-Neumann type operator $DN$. 
Hence, and still at a formal level, we can replace the system \eqref{BL}
by the following system in $\Omega^{bl,-}$: 
\begin{equation} \label{BLbis} 
\left\{
\begin{aligned}
 -\Delta v + \na q   = 0, \quad &  y \in \Omega^{bl,-}, \\
 \na \cdot v  = 0, \quad  & y \in \Omega^{bl,-}, \\
(D(v) \nu)_\tau   = - (D((y_2, 0)) \nu)_\tau, \quad & y \in \pa \Omega^{bl,-}, \\
v \cdot \nu = - (y_2, 0) \cdot \nu,  \quad &  y \in \pa \Omega^{bl,-}, \\
 (-2D(v)e_2 + q \, e_2)\vert_{\{y_2=0\}}  \: = & \: DN(v\vert_{\{y_2=0\}}). 
\end{aligned}
\right.
\end{equation}
A rigorous version of these formal arguments is contained in the next proposition 
\begin{proposition}  {\bf (Equivalent formulation of \eqref{BL})} \label{propertiesstokes} 
\begin{description} 
\item[i)] {\em (Stokes problem in a half-plane)}

\smallskip 
\noindent
 For all $v_0 \in H^{1/2}_{uloc}(\R)$ there exists a unique solution
 $v \in H^1_{loc}(\overline{\R^2_+})$ of \eqref{Stokes} satisfying  
\begin{equation} \label{h1uloc}
 \sup_{k \in \Z} \int_k^{k+1} \int_0^{+\infty} |\na v |^2 \, dy_2 dy_1 < +\infty.
\end{equation}
\item[ii)] {\em (Dirichlet-to-Neumann operator)}

\smallskip \noindent
There is a unique operator 
$$ DN : H^{1/2}_{uloc}(\R) \: \mapsto \: {\cal D}'(\R)$$ 
that satisfies, for all $v_0 \in H^{1/2}_{uloc}(\R)$, and all $\varphi \in
C^{\infty}_c(\overline{\R^2_+})$ with $\na \cdot \varphi = 0$, 
\begin{equation} \label{variationalDN}
 2\int_{\R^2_+}  D(v) \cdot D(\varphi) \: = \: < DN(v_0), \: 
\varphi\vert_{\{y_2=0\}} >. 
\end{equation}
where $v$ is the solution of \eqref{Stokes}. Moreover, for all $v_0\in H^{1/2}_{uloc}(\R)$, the operator $DN(v_0)$ can be extended to a continuous
linear form over the space  $\dis H^{1/2}_c(\R)$ of $H^{1/2}$
functions with compact support.
 \item[iii)] {\em (Transparent boundary condition)}

\smallskip \noindent
Let $(v,q)$ be  a solution of \eqref{BL} in
$H^1_{loc}(\overline{\Omega^{bl}})$
with  $  \sup_{k} \: \int_{\Omega^{bl}_{k,k+1}} | \na v |^2  <
+\infty$. 
Then,  it satisfies \eqref{BLbis}.

\smallskip \noindent
\noindent 
Conversely, let $v^-$  in
$H^1_{uloc}(\Omega^{bl,-})$  be a  solution of \eqref{BLbis}.
Then, the field $v$ defined by
$$ v \: := \:  v^- \mbox{ in } \Omega^{bl,-}, \quad v \: := \: \int_\R
G(y_1-t, y_2) \,  v^-(t,0) \, dt \: \mbox{ for } y_2 > 0 $$
is a solution of \eqref{BL} in $H^1_{loc}(\overline{\Omega^{bl}})$
such that $  \sup_{k} \: \int_{\Omega^{bl}_{k,k+1}} | \na v |^2  <
+\infty$. 
\end{description}
\end{proposition}
{\em Proof of the proposition.} The proof of the proposition is almost
contained in \cite{GerMas}. The only difference lies in the definition of
the Dirichlet-to-Neumann operator.  In \cite{GerMas}, the full gradient is
used in the definition of $DN$,  instead of its symmetric part. Here,
in order to adapt to the Navier condition at the rough boundary, 
$D(u)\nu$ substitutes to  $\pa_\nu u$, and, subsequently,  \eqref{variationalDN} substitutes to 
the relation
$$ \int_{\R^2_+} \na v \cdot \na \varphi \: = \: < DN(v_0), \: 
\varphi\vert_{\{y_2=0\}} >. $$
used in \cite{GerMas}. As these minor changes do not play any
serious role, we skip the proof. 

\bigskip
{\em Step 2.} By the previous proposition, in order to prove well-posedness
of the boundary layer system, we can work with the equivalent system
\eqref{BLbis}. As it is set in a bounded channel, it is amenable to
the kind of the analysis  performed in the previous section, for the
study of the no-slip condition. The keypoint is again to have an
induction relation between the truncated energies. However, the
nonlocal $DN$ operator prevents us from deriving a local relation like
\eqref{ineqEk}. We are able to show the following more complicated
relation: there exists $\eta > 0$ such that, for any $m > 1$, 
\begin{equation} \label{induction}
E_k \: \le \: C_1 \left(  k +
  1 + \frac{1}{m^\eta} \sup_{j\ge k+m} (E_{j+1} - E_j) \:
  + \: m \, \sup_{k+m \ge j \ge k} \left( E_{j+1} - E_j \right) \ \right).
\end{equation}
The same  relation was established in \cite{GerMas}, for the
boundary layer system with a Dirichlet condition. The proof starts
with the the change of unknowns $\:  u := v + (y_2, 0), \: \:  p :=
q$, that turns \eqref{BLbis} into 
\begin{equation*} 
\left\{
\begin{aligned}
 -\Delta u + \na p   = 0, \quad &   y \in \Omega^{bl,-}, \\
 \na \cdot u  = 0, \quad & y \in \Omega^{bl,-}, \\
(D(u) \nu)_\tau   = 0, \quad & y \in \pa \Omega_{bl}, \\
u \cdot \nu =0, \quad &  y \in \pa \Omega_{bl}, \\
 (-2D(u)e_2 + q \, e_2)\vert_{\{y_2=0\}}  \: = \: &  DN(u\vert_{\{y_2=0\}}) - (1,0). 
\end{aligned}
\right.
\end{equation*}
Afterwards, energy estimates are performed, testing against $\chi_k
u$. The only change with respect to 
\cite{GerMas} due the Navier condition is the treatment of the lower
order terms. When a Dirichlet condition holds at the boundary, one
can rely on the Poincar\'e inequality, to obtain 
$$ E_k \: \le \:  C \, \left( \int_{\Omega^{bl,-}_{k,k+1}} | \na u |^2 \: +  \: E_{k+1} - E_k \right)$$    
and then to control $E_k$  from the energy estimate (which gives a
bound on the gradient only). This is no longer possible in the case of the Navier
condition. Moreover, we cannot proceed as in the previous section, 
using the Dirichlet condition at the upper boundary of the
channel. Indeed, in our boundary layer context,  a non-local  condition holds at the upper boundary. This is where
the assumption  (H'') is needed:   it easily implies that 
$$  E_k \: \le \:  C \, \left( \int_{\Omega^{bl,-}_{k,k+1}} |  D(u) |^2 \: +  \: E_{k+1} - E_k \right)$$    
and can therefore be controlled from the energy estimate (which gives
a bound on the symmetric part of the gradient only). From there, all
computations and arguments are similar to those of \cite{GerMas}. We refer
to this paper for all necessary details.  

\subsection{Qualitative behaviour at infinity}
As the solution of \eqref{BL} is now at hand, we still need to show
its convergence to a constant field as $y_2$ goes to infinity. Here, some ergodicity property condition must be
added. We consider the stationary random setting: we take $\omega$
to be an ergodic  stationary random process (on a probability space $(M,\mu)$), obeying the assumptions of Theorem
\ref{Navestimates2}. We then state the following proposition:
\begin{proposition}
There exists $\alpha \in \R$ such that the solution $v$ of \eqref{BL} satisfies
$$ v(y) \rightarrow (\alpha,0), \quad \mbox{ as } y_2 \rightarrow +\infty,
  $$
locally  uniformly in $y_1$, almost surely and in $L^p(M)$
for all finite $p$.
\label{prop:ergo}
\end{proposition} 
This proposition is based on the integral representation
\eqref{poisson}, and the ergodic theorem. It has been proved in
\cite{Basson:2007} (the condition at the rough boundary does not play any role). 

\begin{remark}
It is possible to  derive upper and lower bounds for the ``slip length'' $\alpha$. Such bounds were established by Achdou et al.  \cite{Achdou:1998a} in the periodic setting. They  are still valid in the random stationary  case. Along the lines of \cite{Achdou:1998}, one can prove: {\em for $\om\in\mathcal C^2(\R) \cap W^{2,\infty}(\R)$, 
$$ -Y_{max}\leq \alpha\leq -Y_{min}, \quad  Y_{min}:=\inf_\R \om, \quad Y_{max}:=\sup_\R \om. $$}
In order not to burden the paper, we skip the proof. 
\end{remark}

\section{Estimates for the Navier condition} \label{sec4}

This section is devoted to the proof of Theorem \ref{Navestimates2}. The main novelty lies in the derivation of almost sure  estimates. Indeed, to our knowledge, the previous convergence results dealing with a stationary ergodic setting were all stated in a norm involving an expectation (see \cite{Basson:2007,DGV:2008}). The main steps of the proof are the same as in \cite{GerMas,Basson:2007}: the idea is to build an approximate solution which consists of the main term $u^0$, the boundary layer corrector $v$, and two additional correctors $u^1$ and $r^\eps$. We will review briefly the definition and well-posedness of $u^1$ and $r^\eps$, and focus on the estimates which are required for the proof of Theorem \ref{Navestimates2}.

$\bullet$ We start with some regularity estimates for the function $v$ which solves the boundary layer problem \eqref{BL}:

\begin{lemma}
Let $\beta\in \N^2$ be arbitrary, and let $v$ be the solution of \eqref{BL}. Then for all $a>0$, there exists a constant $C$, depending only on the Lipschitz constant of $\omega$, on $\beta$ and on $a$, such that
$$
\sup_{k\in \Z}\int_k^{k+1}\int_a^\infty\left|\nabla^\beta\nabla v\right|^2\leq C.
$$
In particular, $v \in L^\infty(\R\times (a,\infty))$ for all $a>0$..
\label{lem:regul_est}

\end{lemma}

\begin{proof} 
The arguments are the same as in \cite[Proposition 6]{GerMas}. According to Proposition \ref{propertiesstokes}, $\nabla v \in L^2_{uloc}(\Omega^{bl})$, and $v_0=v_{|\Sigma}\in L^2_{uloc}(\R)$. Since $v$ is given by the representation formula \eqref{poisson} in the upper-half plane, by differentiating under the integral sign in \eqref{poisson}, we obtain
\begin{eqnarray*}
\int_k^{k+1}\int_a^\infty\left|\nabla^\beta\nabla v\right|^2&\leq&\int_k^{k+1}\int_a^\infty\left| \int_{\R} \left| \nabla^\beta\nabla vG (t,y_2) \right|\; \left| v_0(y_1-t)  \right| dt \right|^2dy_1dy_2\\
&\leq&C_\beta  \int_k^{k+1}\int_a^\infty\left| \int_{\R}\frac{1}{t^2 +y_2^2} \left| v_0(y_1-t)  \right| dt \right|^2dy_1dy_2\\
&\leq& C_\beta \int_a^\infty\left(\int_\R\frac{dt}{t^2 + y_2^2}\right)\int_k^{k+1}\int_\R\frac{|v_0(y_1-t)|^2}{t^2+y_2^2}dt\: dy_1 \:dy_2\\
&\leq &C_\beta \|v_0\|_{L^2_\text{uloc}(\R)}\int_a^\infty\int_\R \frac{1}{y_2}\:\frac{1}{t^2+y_2^2}dt\:dy_2 \: \leq \: C_{\beta,a}\|v_0\|_{L^2_\text{uloc}(\R)}.
\end{eqnarray*}

\end{proof}

$\bullet$ We now prove the following result, which is crucial with regards to the derivation of almost sure estimates, and which is the main novelty of this section:

\begin{proposition}

Let $v$ be the solution of the boundary layer system \eqref{BL}. Then the following estimates hold almost surely as $\eps\to 0$:
$$
\begin{aligned}
\sup_{R\geq 1}\frac{1}{R^{1/2}}\left\| v(\cdot/\eps) - (\alpha,0) \right\|_{L^2(\Omega_R)}=o(1),\\
 \sup_{R\geq 1}\frac{1}{R^{1/2}}\left(\left\|\int_0^{x_1} v_2\left(\frac{x_1'}{\eps}, \frac{1}{\eps}\right)dx_1'\right\|_{H^3(-R,R)} + \left\| v_1 \left(\frac{x_1}{\eps}, \frac{1}{\eps}\right)-\alpha \right\|_{H^3(-R,R)}\right)=o(1).
\end{aligned}
$$
\label{prop:as_est}

\end{proposition}

\begin{remark}
This result combines two main ingredients: the deterministic construction of the preceding section, which eventually led to Lemma \ref{lem:regul_est}, and the almost sure convergence of the corrector $v$ in the stationary ergodic setting (see Proposition \ref{prop:ergo}).	We emphasize that both items are important here. In particular, it does not seem possible to prove almost sure estimates by using a probabilistic construction of the boundary layer as in \cite{Basson:2007}.

\end{remark}

\begin{proof}
 
We use an idea developed by Souganidis (see \cite{Souganidis}). Let $\delta>0$ be arbitrary. Then, according to Egorov's Theorem, there exists a measurable set $M_\delta\subset M$ and a number $y_\delta>0$ such that
$$
\begin{aligned}
\left| v(0,y_2,m)-(\alpha,0)\right|\leq \delta\quad \forall m\in M_\delta,\ \forall y_2>y_\delta,\\
P(M_\delta^c)\leq \delta.
\end{aligned}
$$
Without loss of generality, we assume that $y_\delta\geq 1$.

Now, according to Birkhoff's ergodic Theorem, for almost every $m$ there exists $k_\delta>0$ such that if $k> k_\delta$,
$$
A_\delta \: = \: A_\delta(m) \: := \: \left\{ y_1\in \R, \tau_{y_1}m \in M_\delta \right\} \quad \mbox{ satisfies:} \quad   \left| A_\delta \cap (-k,k) \right|\geq 2k (1-2\delta).
$$

For all $R\geq 1,\  \eps>0$, we have
\begin{eqnarray*}
R^{-1} \left\| v(\cdot/\eps) - (\alpha,0) \right\|_{L^2(\Omega_R)}^2&=&\frac{\eps^2}{R}\int_{-R/\eps}^{R/\eps}\int_{\omega(y_1)}^{1/\eps}\left|v(y,m)-(\alpha,0)  \right|^2\:dy\\
&=&\frac{\eps^2}{R}\int_{-R/\eps}^{R/\eps}\int_{\omega(y_1)}^{1/\eps}\left|v(0,y_2,\tau_{y_1}m)-(\alpha,0)  \right|^2\:dy\\
&=&\frac{\eps^2}{R}\int_{-R/\eps}^{R/\eps}\int_{\omega(y_1)}^{y_\delta}\left|v(0,y_2,\tau_{y_1}m)-(\alpha,0)  \right|^2\:dy\\
&+&\frac{\eps^2}{R}\int_{-R/\eps}^{R/\eps}\int_{y_\delta}^{1/\eps}\mathbf 1_{ \tau_{y_1}m \in M_\delta}\left|v(0,y_2,\tau_{y_1}m)-(\alpha,0)  \right|^2\:dy\\
&+& \frac{\eps^2}{R}\int_{-R/\eps}^{R/\eps}\int_{y_\delta}^{1/\eps}\mathbf 1_{ \tau_{y_1}m \in M_\delta^c}\left|v(0,y_2,\tau_{y_1}m)-(\alpha,0)  \right|^2\:dy\\
&=&\sum_{j=1}^3 I_j
\end{eqnarray*}
We have clearly, by definition of $A_\delta$,
$$
\begin{aligned}
I_1\leq  \eps \sup_{R'\geq 1}\frac{1}{R'}\left\| v(y)-(\alpha,0) \right\|^2_{L^2(((-R',R')\times(-1,y_\delta))\cap \Omega^{bl})}\leq C_\delta \eps,\\
I_2\leq \frac{\eps^2}{R}\frac{2R}{\eps}\frac{1}{\eps}\delta^2\leq 2 \delta^2.
\end{aligned}
$$
Notice that the constant $C_\delta$ in the first inequality depends on the random parameter $m$.

As for the third integral, recall that $v\in L^\infty(\R\times(1,\infty))$ according to Lemma \ref{lem:regul_est} and that $y_\delta\geq 1$. Thus we have, for all $R\geq 1$ and if $\eps<1/k_\delta$,
$$
I_3\leq 4\delta \left( |\alpha|^2 + \| v\|_{L^\infty(\R\times(1,\infty))}^2 \right).
$$
Gathering the three terms, we deduce that the first estimate of the Lemma holds true.

Concerning the second estimate, we define the vector field $u(y)=v(y)+(y_2, 0)$. Notice that $u$ is divergence free and that $u\cdot \nu=0$ at the lower boundary of $\Omega^{bl}$. Consequently,
following \cite{Basson:2007}, Proposition 14, we write
\begin{eqnarray}
 \int_0^{x_1} v_2\left(\frac{x_1'}{\eps}, \frac{1}{\eps}  \right)dx_1'&=&\eps \int_0^{x_1/\eps}u_2\left(y_1, \frac{1}{\eps}  \right)dy_1\nonumber\\
&=&\eps\int_{\omega(x_1/\eps)}^{1/\eps}u_1\left(\frac{x_1}{\eps},y_2\right)dy_2-\eps\int_{\omega(0)}^{1/\eps}u_1\left(0,y_2\right)dy_2\nonumber\\
&=&\eps\int_{\omega(x_1/\eps)}^{1/\eps}(v_1-\alpha)\left(\frac{x_1}{\eps},y_2\right)dy_2-\eps\int_{\omega(0)}^{1/\eps}(v_1-\alpha)\left(0,y_2\right)dy_2\label{in:v2}\\
&-&\eps\left[\frac{\omega^2\left(\frac{x_1}{\eps}  \right) - \omega^2(0)}2 + \alpha \left(\omega\left(\frac{x_1}{\eps }\right) - \omega(0)\right)\right].\nonumber
\end{eqnarray}
The last term is bounded in $L^\infty$ by $\eps(\|\omega\|_\infty^2 +2 \|\omega\|_\infty)$, and thus converges towards zero in the appropriate norm. As for the other two terms, set
$$
U^\eps(y_1):= \eps \int_{\omega(y_1)}^{1/\eps} (v_1-\alpha)(y_1, y_2) dy_2.
$$
Using the same decomposition as previously, we write, for $\delta>0$ arbitrary,
\begin{eqnarray*}
U^\eps(y_1)&=& \eps \int_{\omega(y_1)}^{y_\delta }(v_1-\alpha) + \eps \mathbf 1_{\tau_{y_1}\omega\in A_\delta} \int_{y_\delta}^{1/\eps}(v_1-\alpha) + \eps \mathbf 1_{\tau_{y_1}\omega\in A_\delta^c} \int_{y_\delta}^{1/\eps}(v_1-\alpha),
\end{eqnarray*}
and thus
$$
|U^\eps(y_1)|\leq \eps \int_{\omega(y_1)}^{y_\delta }|v_1-\alpha| + \delta + \|v_1-\alpha\|_{L^\infty(\R\times(1,\infty))}\mathbf 1_{\tau_{y_1}\omega\in A_\delta^c }.
$$
Consequently, we obtain, for all $R\geq 1$ and for $\eps$ small enough (depending on $\delta$),
\begin{eqnarray*}
\frac{1}{R}\int_{-R}^R\left| U^\eps\left(\frac{x_1}{\eps} \right) \right|^2dx_1&\leq&C\frac{\eps^3}{R}(y_\delta-\inf\om)\int_{-R/\eps}^{R/\eps}\int_{\omega(y_1) }^{y_\delta}|v_1-\alpha|^2\\
&&+C\left( \delta^2 + 4\delta \|v_1-\alpha\|_{L^\infty(\R\times(1,\infty))}^2\right)\\
&\leq & C_\delta \eps^2 + C \delta.
\end{eqnarray*}
Hence, as $\eps$ vanishes,
$$
\sup_{R\geq 1} \frac{1}{R}\int_{-R}^R\left| U^\eps\left(\frac{x_1}{\eps}\right)  \right|^2dx_1 =o(1).
$$

The second  term in \eqref{in:v2} is easily treated: since it does not depend on $x_1$, we have
\begin{eqnarray*}
&&\sup_{R\geq 1} \frac{1}{R}\int_{-R}^R\left| \int_{\omega(0)}^{1/\eps}(v_1-\alpha)\left(0,y_2\right)dy_2  \right|^2dx_1\\
&=&\left| \int_{\omega(0)}^{1/\eps}(v_1-\alpha)\left(0,y_2\right)dy_2  \right|^2 .
\end{eqnarray*}
Since $(v_1-\alpha)(0, y_2)$ vanishes almost surely as $y_2\to \infty$, we infer that
$$
\lim_{\eps\to 0} \eps\int_{\omega(0)}^{1/\eps}(v_1-\alpha)\left(0,y_2\right)dy_2=0\quad \text{a.s.}
$$
This proves that
$$
\sup_{R\geq 1}\frac{1}{R^{1/2}}\left\|\int_0^{x_1} v_2\left(\frac{x_1'}{\eps}, \frac{1}{\eps}\right)dx_1'\right\|_{L^2(-R,R)} =o(1)
$$
as $\eps$ vanishes.

The estimates
$$
\begin{aligned}
 \sup_{R\geq 1}\frac{1}{R^{1/2}}\left\|v_2\left(\frac{x_1}{\eps}, \frac{1}{\eps}\right)\right\|_{L^2(-R,R)}=o(1),\\
 \sup_{R\geq 1}\frac{1}{R^{1/2}} \left\| v_1 \left(\frac{x_1}{\eps}, \frac{1}{\eps}\right)-\alpha \right\|_{L^2(-R,R)}=o(1)
\end{aligned}
$$
are derived in a similar fashion. There remains to prove that for $1\leq k\leq 3$
$$
  \sup_{R\geq 1}\frac{1}{R^{1/2}}\eps^{-k}\left\| (\pa_{y_1}^k v)\left(\frac{x_1}{\eps}, \frac{1}{\eps}\right)\right\|_{L^2(-R,R)}=o(1).
$$

Notice that it suffices to prove that for $k=1,2,3$,
\be\label{conv_derivees}
\lim_{y_2\to\infty} y_2^{k}\pa_{y_1}^k v(0,y_2, m)=0\quad\text{almost surely for }m\in M.
\ee
Then the same arguments as above allow us to conclude.

In order to obtain \eqref{conv_derivees}, we use the same estimates as in \cite{Basson:2007}, Proposition 13. We write
$$
\pa_{y_1}^k v(0,y_2, m)=\int_\R y_1'\pa_{y_1}^{k+1} G(-y_1',y_2)\left(\frac{1}{y_1'}\int_0^{y_1'}( v_0(-z, m) - (\alpha,0))dz\right)dy_1'.
$$
Since $G$ is homogeneous of degree $-1$, it can be easily proved that for $k\geq 1$,
$$
\begin{aligned}
\left|\pa_{y_1}^{k} G(y_1,y_2) \right|\leq C_k  (y_1^2 + y_2^2)^{-\frac{k+1}{2}},\\
\left|y_1\pa_{y_1}^{k+1} G(y_1,y_2) \right|\leq C_k  (y_1^2 + y_2^2)^{-\frac{k+1}{2}}.
\end{aligned}
$$
Now, let $\delta>0$ be arbitrary. Almost surely, there exists $y_\delta>0$ (depending on $m$) such that
$$
\left|\frac{1}{y_1}\int_0^{y_1}(v_0(z,m)-(\alpha, 0))dz\right|\leq \delta\quad\text{if }|y_1|\geq y_\delta.
$$
As a consequence,
\begin{eqnarray*}
&&\left|\int_{|y_1'|\geq y_\delta}y_1'\pa_{y_1}^{k+1} G(-y_1',y_2)\left(\frac{1}{y_1'}\int_0^{y_1'}( v_0(-z, \omega) - (\alpha,0))dz\right)dy_1'\right|\\ 
&\leq&C_k\delta\int_{|y_1|\geq y_\delta}(y_1^2 + y_2^2)^{-\frac{k+1}{2}}dy_1\\
&\leq&C_k\frac{\delta}{y_2^k}.
\end{eqnarray*}
On the other hand, 
\begin{eqnarray*}
 &&\left|\int_{|y_1'|\leq y_\delta}y_1'\pa_{y_1}^{k+1} G(-y_1',y_2)\left(\frac{1}{y_1'}\int_0^{y_1'}( v_0(-z, m) - (\alpha,0))dz\right)dy_1'\right|\\ 
&\leq&C_k \int_0^{y_\delta}\left| ( v_0(-z, m) - (\alpha,0)) \right|dz \int_{|y_1|\leq y_\delta}   (y_1^2 + y_2^2)^{-\frac{k+2}{2}}dy_1\\
&\leq &C_\delta \frac{1}{y_2^{k+1}}.
\end{eqnarray*}
Gathering the two terms, we infer that for all $\delta>0$, there exists $C_\delta>0$ such that
$$
\left| y_2^k\pa_{y_1}^k v(0,y_2, m) \right|\leq \delta + \frac{C_\delta}{y_2}\quad \forall y_2>0,
$$
and thus the quantity in the left-hand side vanishes almost surely as $y_2\to\infty.$

\end{proof}

$\bullet$ We are now ready to prove the convergence result stated in Theorem \ref{Navestimates2}. Following \cite{GerMas}, we set 
$$
\ueapp(x):=u^0(x)+6\phi\eps v\left(\frac{x}{\eps}  \right) + \eps u^1 (x) + \eps \re(x) + 6\phi \eps^2 v^1\left(\frac{x}{\eps}  \right)
,
$$
where the correctors $u^1$ and $\re$ ensure that $\ueapp$ satisfies the Dirichlet boundary condition at the upper boundary and the zero flux condition. The term $v^1$ is a boundary layer term which compensates the tangential trace of $u^0+6\phi \eps v$ at the rough boundary\footnote{It can be checked that this extra boundary layer term is needed only when the original slip length $\lambda^\eps$ is such that $\lambda \lesssim 1.$}.   
Additionnally, $u^1$ is intended to be $O(1)$ while $\re=o(1)$. 

$\vartriangleright$ We choose $u^1$ to be the solution of
$$
\left\{
\begin{array}{l}
 -\Delta u^1 + u^0\cdot \nabla u^1 + u^1\cdot \nabla u^0+\nabla p^1=0, \quad x\in \Omega,\\
\nabla\cdot u^1=0, \quad x\in \Omega,\\
u^1_{|x_2=0}=0,\quad u^1_{|x_2=1}=(-6\phi\alpha,0),\\
\int_\sigma u^1_1=-6\phi\alpha.
\end{array}
\right.
$$
Notice that we assume that $u^1$ satisfies a no-slip condition at the lower boundary. This stems from the non-degeneracy of the frontier $\Gae$: in order that the non-penetration condition is satisfied at order $\eps$, $u^1$ must vanish at $x_2=0$. We recall that the same argument led to the no-slip condition for $u^0$ at $x_2=0$.

Hence the vector field $u^1$ is exactly the same as in \cite{GerMas} and is a combination of Couette and Poiseuille flows:
$$
u^1_1(x)=6\phi(-4\alpha x_2+3\alpha x_2^2),\quad u^1_2(x)=0,\quad x\in\Omega
$$
and we extend $u^1$ by zero outside $\Om$.

$\vartriangleright$ The additional boundary term $v^1$ solves the system
$$
\left\{
\begin{array}{l}
-\Delta v^1 + \na q_1=0 \quad\text{in } \Om^{bl} \\
v^1\cdot \nu=(y_2^2, 0)\cdot\nu \quad \text{on } \Gamma, \\ 
\lambda^0(D(v^1)\nu)_\tau=(v+(y_2, 0))_\tau+ \lambda^0(D((y_2^2, 0))\nu)_\tau\:\: \text{on } \Gamma, 
\end{array} 
\right.
$$
under the condition 
$$
\sup_{k\in\Z}\int_{\Om^{bl}_{k,k+1}}|\na v^1|^2.
$$
Using the same techniques as in Section \ref{sec4}, energy estimates in $H^1_{uloc}(\Om^{bl}_- )$ for $v^1$ can be proved, leading to existence and uniqueness of $v^1$. Additionally, there exists $\beta\in\R$ such that
$$
\lim_{y_2\to\infty}v^1(y_1,y_2)=(\beta, 0)
$$ 
almost surely.

$\vartriangleright$ As for $\re$, we use the following Lemma:
\begin{lemma}
There exists a vector field $\re \in H^2_\text{loc}(\Omega)$ satisfying
$$
\left\{
\begin{array}{l}
\nabla\cdot \re =0,\\
\re_{|x_2=0}=0,\quad \re_{|x_2=1}=6\phi\left[(\alpha, 0) - v\left(\frac{x_1}{\eps},\frac{1}{\eps} \right) - \eps v^1\left(\frac{x_1}{\eps}, \frac{1}{\eps}  \right)\right] ,\\
\int_\sigma \re_1=-\int_{\sigma^\eps\setminus\sigma}u^0_1 - 6\phi\int_{\sigma_\eps} (v_1+ \eps v^1_1)(x/\eps) +6\phi\alpha,
\end{array}
\right.
$$
and such that
$$
\sup_{R\geq 1}\frac{1}{R^{1/2}}\| \re\|_{H^2(\Omega_{-R,R})}=o(1)\text{ and }\|\re\|_{W^{2,\infty}(\Omega)}=O(1).
$$
\label{lem:re}
\end{lemma}

The Lemma follows directly from Proposition \ref{prop:as_est} and the construction in Proposition 5.1 in \cite{Basson:2007}. Once again, we extend $\re$ by zero outside $\Om$.

By construction, the function $\ueapp$ satisfies
$$\left\{
\begin{array}{l}
\dis -\Delta \ueapp + \ueapp\cdot \nabla \ueapp + \nabla p^\eps_\text{app}=\dv G^\eps + f^\eps\quad\text{in }\Om^\eps\setminus\Sigma,\\
\ueapp\vert_{x_2=1}=0,\\
\ueapp\cdot \nu\vert_{\Gamma^\eps}= 0 ,\\
(\ueapp)_\tau\vert_{\Gamma^\eps}  \:    =  \:  \lambda^0  (  D(\ueapp)\nu)_\tau\vert_{\Gamma^\eps} + g^\eps ,\\
\dis [\ueapp]\vert_{\Sigma}=0,\quad \left[D(\ueapp)e_2 -  p^\eps e_2 \right]\vert_{\Sigma}=D(\eps \re + \eps u^1)e_2\vert_{\Sigma}=:\varphi^\eps
\end{array}
\right.
$$
where
\begin{eqnarray*}
G^\eps&=&  \eps u^0\otimes\left(6\phi\left(v\left(\frac{\cdot}{\eps}  \right)-(\alpha , 0) + \eps v^1 \left(\frac{\cdot}{\eps}  \right)\right)+ \re \right) \\&+& \eps 	\left(6\phi\left(v\left(\frac{\cdot}{\eps}  \right)-(\alpha , 0) + \eps v^1 \left(\frac{\cdot}{\eps}  \right)\right)+ \re \right)\otimes u^0\\
&+& \eps^2 \left(6\phi\left(v\left(\frac{\cdot}{\eps}  \right) + \eps v^1 \left(\frac{\cdot}{\eps}  \right)\right)+ \re \right)\otimes \left(6\phi\left(v\left(\frac{\cdot}{\eps}  \right) + \eps v^1 \left(\frac{\cdot}{\eps}  \right)\right)+ \re \right)
\end{eqnarray*}
and 
$$f^\eps=-\eps \Delta \re, \quad g^\eps=6\phi(\eps^2 v^1(x/\eps) -  (x_2^2, 0))_\tau\vert_{\Gamma^\eps}. $$

According to the estimates of Section \ref{sec3}, Proposition \ref{prop:as_est} and Lemma \ref{lem:re}, we have
$$
\begin{aligned}
\sup_{R\geq 1}\frac{1}{R^{1/2}}\| G^\eps\|_{L^2(\Omega^\eps_R)}=o(\eps),\\
\sup_{R\geq 1}\frac{1}{R^{1/2}}\| f^\eps\|_{L^2(\Omega^\eps_R)}=o(\eps),\\
\sup_{R\geq 1}\frac{1}{R^{1/2}}\|\varphi^\eps\|_{L^2(\Sigma_R)}=o(\eps) + O(\eps\phi),\\
\sup_{k\in\Z}\|g^\eps\|_{L^2(\Gamma^\eps_{k,k+1})}=O(\eps^2).
\end{aligned}
$$

Consequently, setting $\we=\ue-\ueapp$, we obtain
\be\label{eq:we}
 - \Delta \we + (\ue \cdot \nabla)\we + (\we \cdot \nabla)\ueapp + \nabla \qe = -\dv G^\eps - f^\eps\quad \text{in}\Om^\eps\setminus \Sigma,
\ee
and $\we$ satisfies the same boundary and jump conditions as $\ueapp$.

The next step is to derive energy estimates for the above system. The proof goes along the same lines as the one in Section \ref{sec2}, and therefore, we skip the details. The main steps are the following:
\begin{enumerate}
	\item First, we derive an energy estimate in $\Om^\eps_n$ for a sequence $(\we_n)_{n\in\N}$ satisfying \eqref{eq:we} in $\Om^\eps_n$ with homogeneous Dirichlet boundary conditions at $x_1=\pm n$; more precisely, we prove that for $\phi$ small enough,
$$
\int_{\Om^\eps_n} (|D(\we_n)|^2 + |\we_n|^2 + |\na \we_n|^2)= n o(\eps^2).
$$

	\item By induction on $k$, we prove that for all $n\geq 1$, $k\in\{1,\cdots,  n-1\}$,
$$
\int_{\Om^\eps_k} (|D(\we_n)|^2 + |\we_n|^2 + |\na \we_n|^2)= k o(\eps^2).
$$
	\item Passing to the limit as $n\to\infty$, we deduce that for all $k\geq 1,$
$$
\int_{\Om^\eps_k} (|D(\we)|^2 + |\we|^2 + |\na \we|^2)= k o(\eps^2).
$$

\end{enumerate}

There are two main differences with the estimates of Section \ref{sec2}.  The first one lies in terms of the type
$$
\int_{\Om^\eps_k}\left| \left((\we_n \cdot \nabla)\ueapp\right)\cdot \we_n \right|;
$$
indeed, because of the boundary layer term $v$, $\na \ueapp$ does not belong to $L^\infty(\Om^\eps)$ in general. Therefore, using Sobolev embeddings, we have 
\begin{eqnarray*}
	\int_{\Om^\eps_k} |\we_n|(x)^2 \left|\na v\left(\frac{x}{\eps}  \right)  \right|\: dx&=&\sum_{j=-k}^{k-1} 	\int_{\Om^\eps_{j,j+1}} |\we_n(x)|^2 \left|\na v\left(\frac{x}{\eps}  \right)  \right|\: dx\\
&\leq & \eps\sum_{j=-k}^{k-1} 	\left(\int_{\Om^\eps_{j,j+1}}|\we_n|^4\right)^{1/2}\left(\int_{\frac{j}{\eps}}^{\frac{j+1}{\eps}}\int_{\om(y_1)}^{1/\eps}|\na v(y)|^2\:dy\right)^{1/2}\\
&\leq & C\sqrt{\eps}\|\na v\|_{H^1_{uloc}(\Om^{bl})}\|\we_n\|_{H^1(\Om^\eps_k)}^2\leq C\sqrt{\eps}\phi\|\we_n\|_{H^1(\Om^\eps_k)}^2.
\end{eqnarray*}

The second difference comes from the boundary term $g^\eps$, namely
$$
\frac{1}{\lambda^0}\int_{\Gamma^\eps_n}\chi_k(\we_n)_\tau g^\eps\leq \frac{1}{2\lambda^0}\int_{\Gamma^\eps_n}\chi_k|(\we_n)_\tau|^2 + C(k+1) \frac{\eps^4}{\lambda^0}.
$$
The first term of the right-hand side can be absorbed in the boundary term coming from the integration by parts of $\int\Delta \we_n\cdot \we_n$. The second one is clearly $o((k+1)\eps^2).$ Notice that this is the reason why we need the additional boundary layer term $v^1$: if we merely take 
$$
\ueapp= u^0 + \eps v \left(\frac{x}{\eps}  \right) + \eps u^1(x) + \eps \re(x)
$$
then $g^\eps=O(\eps)$, and the second term in the right-hand side of the preceding inequality is $O((k+1)\eps^2/\lambda^0)$.

The inequality relating $E_k$ and $E_{k+1}$ (with the same notation as in Section \ref{sec2}) becomes in the present case
\begin{eqnarray*}
E_k&\leq &C\left( E_{k+1}-E_k + \sqrt{\eps\phi} E_{k+1} + \eps \eta(\eps) \sqrt{k+1}\sqrt{E_{k+1}}+ (k+1)\eps^4 \right)	\\
&+& C (\phi + \eps \eta(\eps)\sqrt{k+1})(E_{k=1}-E_k)^{3/2},
\end{eqnarray*}
for some function $\eta$ such that $\lim_{0⁺^+}\eta=0.$ By induction we infer easily that
$$
E_k\leq k\eps^2 \eta_1(\eps),
$$
for some other function $\eta_1$ vanishing at zero, which completes the second step described above. The two other steps are left to the reader.

We infer that 
$$
\sup_{R\geq 1}\frac{1}{R^{1/2}}\left\| \ue - \ueapp\right\|_{H^1(\Om^\eps_R)} = o(\eps)\quad\text{almost surely.}
$$

On the other hand, let $u^N$ be the solution of \eqref{NS}-\eqref{Navier} with $\lambda=6\phi\alpha \eps$. Then the function $u^N$ is explicit: as in \cite{GerMas}, we have
$$
u^N=(6\phi U^N(x_2), 0)\quad\text{with }U_N(x_2)=-\frac{1+\eps\alpha}{1+4\eps\alpha}x_2^2 +\frac{1}{1+4\eps\alpha}x_2+\frac{\eps\alpha}{1+4\eps\alpha},
$$
so that
$$
u^N= u^0 + 6\phi\eps(\alpha, 0) + \eps u^1 + O(\eps^2)\quad\text{in }L^2_{uloc}(\Omega).
$$
From there, we obtain
$$
\sup_{R\geq 1}\frac{1}{R^{1/2}}\left\| u^N- \ueapp\right\|_{L^2(\Om^\eps_R)} = o(\eps)\quad\text{almost surely.}
$$
Theorem \ref{Navestimates2} follows.

{\small
\def\cprime{$'$} \def\cprime{$'$} \def\cprime{$'$}

\end{document}